\documentclass[11pt,reqno]{amsart}
\usepackage{graphicx,amscd}
\usepackage{amsfonts,amsmath,amssymb}
\newcommand{\rl}{{\mathbb{R}}}
\newcommand{\cx}{{\mathbb{C}}}

\newcommand{\codim}{\rm codim\,}

\newcommand{\e}{\varepsilon}
\newcommand{\supp}{{\mathrm{supp}}}
\newcommand{\tmop}[1]{\ensuremath{\operatorname{#1}}}
\renewcommand{\Re}{\tmop{Re}}
\renewcommand{\Im}{\tmop{Im}}

\newcommand{\B}{\mathbb{B}}
\newcommand{\D}{\mathbb D}
\renewcommand {\a}{\alpha}
\renewcommand {\b}{\beta}
\renewcommand {\supp}{{\rm supp\,}}
\newtheorem{theorem}{Theorem}[section]
\newtheorem{lemma}[theorem]{Lemma}
\newtheorem{prop}[theorem]{Proposition}
\newtheorem{cor}[theorem]{Corollary}

\theoremstyle{definition}
\newtheorem{defn}[theorem]{Definition}
\newtheorem{rem}[theorem]{Remark}
\newtheorem{ex}[theorem]{Example}

\begin{document}
\title{Polynomially convex hulls of  singular real manifolds}
\author{Rasul Shafikov* and Alexandre Sukhov**}
\begin{abstract}
We obtain local and global results on polynomially convex hulls of Lagrangian and totally real
submanifolds of $\cx^n$ with self-intersections and open Whitney umbrella points.
\end{abstract}

\maketitle

\let\thefootnote\relax

* Department of Mathematics, the University of Western Ontario, London, Ontario, N6A 5B7, Canada,
e-mail: shafikov@uwo.ca. The author is partially supported by the Natural Sciences and Engineering 
Research Council of Canada.

**Universit\'e des Sciences et Technologies de Lille, 
U.F.R. de Math\'ematiques, 59655 Villeneuve d'Ascq, Cedex, France,
e-mail: sukhov@math.univ-lille1.fr. The author is partially supported by Labex CEMPI.

\tableofcontents


\section{Introduction}
This paper is concerned with polynomially convex hulls of $n$-dimensional real manifolds in $\cx^n$
with singularities of special type: self-intersections and open Whitney umbrellas. Our motivation comes
from the work of Givental~\cite{G} where he proved that any compact real surface can be realized as
a Lagrangian submanifold of $\cx^2$ with isolated singularities of this type. More precisely, following
Arnold \cite{Ar}, by a {\it Lagrangian inclusion} we mean a smooth mapping from a manifold into a 
symplectic manifold which is Lagrangian embedding in a neighbourhood of almost every point. Givental 
proved that every compact real surface admits a Lagrangian inclusion into $\cx^2$ with a finite number $d$ 
of transverse double self-intersection points, and $u$ open Whitney umbrellas, satisfying the following 
topological formula (mod 2 for a nonorientable $L$)
$$
L \cdot L = \chi(L) + 2 d + u,
$$
where $L \cdot L$ is the self-intersection index of the fundamental cycle of $L$ in $H_2(\cx^2)$ and 
$\chi(L)$ is the Euler characteristic. According to Givental~\cite{G} and Ishikawa \cite{Is}, for a generic 
Lagrangian inclusion double self-intersections and open Whitney umbrellas are the only singularities. 
Moreover, they are stable under Lagrangian deformations.

Since Gromov's~\cite{Gr} work it is understood that topological properties of 
Lagrangian inclusions into $\cx^n$ are related to the complex structure. For instance, (non)-existence of 
nonconstant Riemann surfaces
with the boundary glued to a Lagrangian inclusion has many topological consequences. This problem in its
turn can be regarded from the complex analysis point of view as a question on polynomial  convexity.
Polynomial convexity properties of real submanifolds in a complex manifold are of fundamental importance in complex
analysis and have been studied by many authors, we refer to a recent monograph of Stout~\cite{St} 
dedicated to this subject. A considerable progress in the case of Lagrangian and totally real submanifolds was made in the works of Alexander \cite{Al0}, Bedford-Klingenberg \cite{be-kl}, Duval-Sibony \cite{DuSi1,DuSi2}, Forstneri\v c-Stout \cite{FoSt},  Forstneri\v c -Rosay \cite{FoRo}, Gromov \cite{Gr},  Ivashkovich-Shevchishin \cite{IvSh}, J\"oricke \cite{Jo2}, Kenig-Webster \cite{KW} and other authors. However, little is known about polynomial convexity properties of singularities of Lagrangian
inclusions. This is the main subject of this paper, which is continuation of our previous work in~\cite{SS}.

Denote by $z = x + iy$ and $w = u + iv$ the standard complex coordinates in $\cx^2$. Consider the map

\begin{eqnarray}
\label{StandUmbr}
\pi: \rl^2_{(t,s)} \ni (t,s) \mapsto (ts,\frac{2t^3}{3},t^2,s) \in \rl^{4}_{(x,u,y,v)}
\end{eqnarray}
The image $\Sigma:= \pi(\rl^2)$ is called {\it the standard (unfolded or open) Whitney umbrella}. Denote by 
$\B_n(p,\e)$ the Euclidean ball in $\cx^n$ centred at $p$ and of radius $\e > 0$, and shortly write $\B_n$  
if $p= 0$ and $\e = 1$. We also often drop the index $n$ indicating the dimension when its value is clear from 
the context. Note that the map $\pi$ has the only critical point at the origin; furthermore, $\pi$ is a 
homeomorphism between  neighbourhoods of the origin in $\rl^2$ and in $\Sigma$.

\begin{defn} 
A closed subset $X$ of $\cx^n$ is called locally polynomially convex near a point $p \in X$ if for every sufficiently small $\e> 0$ the intersection $X \cap \overline{\B_n(p,\e)}$ is polynomially convex.
\end{defn} 

Denote by $\omega = dx \wedge dy + du \wedge dv$ the 
standard symplectic form on $\cx^2$. A smooth map $\phi: \cx^2 \to \cx^2$ is called {\it symplectic} if $\phi^*\omega = \omega$. Such a map   is necessarily a local diffeomorphism, so we call it  a (local) {\it  symplectomorphism}. 
Our first result is the following

\begin{theorem}\label{t:ga}
Suppose that $\phi: \cx^2 \to \cx^2$ is   a smooth  generic symplectomorphism near the origin. Then 
the surface $\phi(\Sigma)$ is locally polynomially convex near $\phi(0)$.
\end{theorem}

When $\phi$ is real analytic, this result is obtained in \cite{SS}. The word ``smooth" everywhere  means of class $C^\infty$ and 
the word ``generic" here means that $j^\infty\phi$, the jet of $\phi$ at zero, does not belong to a closed nowhere dense subset 
of the space 
of jets of symplectic maps at the origin. Our proof pushes further the method developed in \cite{SS}: in order to pass from the real analytic category to the smooth one we use more advanced tools of the singularity theory. Theorem~\ref{t:ga} can be viewed as an analogue of the result 
of Forstneri\v c and Stout \cite{FoSt} on local polynomial convexity of a totally real surface near an isolated 
hyperbolic point.

Our next result establishes local polynomial convexity near the other type of singularity of a Lagrangian 
inclusion--transverse self-intersection.

\begin{theorem}\label{InterTheo}
Let $L_1$ and $L_2$ be smooth Lagrangian submanifolds in $\cx^n$ intersecting transversally at a point $p$. 
Then the union $(L_ 1 \cup L_2)$ is locally polynomially convex near $p$. 
\end{theorem}

The proof of this theorem also shows that for a sufficiently small neighbourhood $U$ of $p$, any continuous function on 
$U\cap (L_ 1 \cup L_2)$ can be approximated by holomorphic polynomials. Similar conclusion holds under the assumptions
of Theorem~\ref{t:ga} near an open umbrella point, since the argument in \cite[Cor. 1]{SS} is valid in the smooth 
case.

Theorems~\ref{t:ga} and~\ref{InterTheo} now state that a generic Lagrangian inclusion in $\cx^2$ is locally polynomially
convex. This leads to some global consequences. In this paper all compact manifolds are without boundary. Denote by $\D = \{ \zeta \in \cx: \vert \zeta \vert < 1 \}$ the unit disc in $\cx$. Recall that  a holomorphic map $f: \D \to \cx^n$ is called an analytic disc; if $f$ is continuous on the closed disc $\overline{\D}$, the restriction $f\vert_{\partial\D}$ of $f$ to the boundary $\partial\D$ is called the boundary of $f$. The boundary of an analytic disc $f$ is said to be attached to a subset $E$ of $\cx^n$ if $f(\partial\D) \subset E$.

\begin{theorem}\label{continuity10}
Suppose that a smooth compact Lagrangian immersion $L$ in $\cx^n$ admits a finite number of self-intersection points  and is locally polynomially convex near every self-intersection point. Then there exists a nonconstant analytic  disc continuous on $\overline{\D}$ with boundary attached to~$L$.
\end{theorem}

For Lagrangian embeddings this result is due to Gromov~\cite{Gr}. Note that in the above theorem we do not 
require self-intersections to be double or transverse. The class of compact $n$-manifolds  admitting Lagrangian immersions in $\cx^n$ is considerably wider that the class of manifolds embeddable as Lagrangian submanifolds. For example, the torus is the only compact orientable surface admitting a Lagrangian embedding in $\cx^2$. By comparison, every compact orientable surface admits a Lagrangian immersion in $\cx^2$.

In view of Theorem~\ref{InterTheo} we have the following

\begin{cor}
\label{continuity20}
Let  $L$ be a smooth compact Lagrangian immersion in $\cx^n$ with  a finite number of double transverse self-intersection points. 
Then there exists a nonconstant analytic  disc  continuous on $\overline{\D}$ with the boundary attached to $L$.
\end{cor}

This corollary is not new. Ivashkovich and Shevchishin \cite{IvSh} proved the existence of an analytic disc $f$ attached to 
an immersed Lagrangian manifold under an assumption of weak transversality which holds for transverse double intersections. Their approach is closer to the original work of Gromov and is based on the general compactness theorem for $J$-complex curves with boundaries glued to a Lagrangian immersion with weakly transverse self-intersections. Their method also works 
for  symplectic manifolds with  certain tamed almost complex structures. 
Our proof, based on Alexander's version \cite{Al} of Gromov's theory, uses purely complex-analytic tools. Note that 
Theorem~\ref{continuity10} also works in some cases when the result of Ivashkovich-Shevchishin cannot be applied. 
One occurrence of this is described in Example \ref{Ex1} of Section~\ref{s.X}. The condition of weak transversality from 
\cite{IvSh} fails there, but Theorem~\ref{continuity10} gives the result. We remark that in~\cite{Gr2} 
Gromov asserts that his method \cite{Gr} can be adapted for the case of any immersed Lagrangian submanifold
of~$\cx^n$. Neither our approach nor \cite{IvSh} gives the result in such generality. Nevertheless, Alexander's method implies  (see Proposition~\ref{GluingDisc2}) without any transversality assumption that the polynomially convex hull of a smooth compact Lagrangian immersion $L$  in $\cx^n$ contains a complex analytic curve of finite area with boundary glued to $L$. One technical tool
used in the proof of Theorem~\ref{continuity10} is Proposition~\ref{t:cont1} which established continuity
of holomorphic discs up to the boundary when the cluster set at the boundary $\partial \mathbb D$ is contained in a singular totally
real set, see Section~\ref{s.5} for details.

Given a compact $K\subset \cx^n$ we denote by $\widehat K$ the polynomially convex hull of $K$. The polynomially
convex hull of a general Lagrangian inclusion is described in the next theorem. This is a consequence of Duval-Sibony's
theory of hulls \cite{DuSi1, DuSi2}, combined with Alexander's technique \cite{Al} adapted to the case of totally real 
immersions.

\begin{theorem}\label{t.hulls}
(i) Let $L$ be a smooth compact Lagrangian surface in $\cx^2$ with a finite set  of  open Whitney umbrellas. 
Then $L\ne \widehat L$, and there exists a  positive (1,1) current $S$ such that $\supp(dd^cS) \subset L$ 
and  $\supp(S)$ is not contained in $L$, but is contained in $\widehat L$.

(ii) Let $E$ be a smooth compact totally real immersed manifold in $\cx^n$. Then $E\ne \widehat E$ and there exists a positive current $S$ of bidimension $(1,1)$ 
and mass 1, with $dS$ supported on $E$, such that $\supp(S)$ is contained in $\widehat E$, but not in $E$.
\end{theorem}

\noindent{\bf Acknowledgment.} The authors are deeply grateful to Sergey Ivashkovich for very useful
discussions concerning Alexander's work. He should be considered as a contributor to Section \ref{s.Adisc}.

\section{Background}

\subsection{Rings of smooth functions and spaces of jets}  
We recall some basic notions of the local  theory of singularities of differentiable maps following 
\cite{AGV,BL,gg}; these classical works contain proofs of all statements of the present subsection.   
Denote by
${\mathcal C}(n,m)$ the ring of $C^\infty$ differentiable germs of functions $(\rl^n,0) \to \rl^m$; we write 
${\mathcal C}(n)$ when $m=1$. Consider
$$
{\mathcal M}(n) = \{ f \in {\mathcal C}(n) \ \vert\ f(0) = 0 \}.
$$
Then ${\mathcal M}(n)$ is the unique maximal ideal in ${\mathcal C}(n)$. Let $x = (x_1,...,x_n)$ be the standard 
coordinates in $\rl^n$. The ideal ${\mathcal M}(n)$ coincides with the ideal $\langle x_1,...,x_n\rangle$ generated by
the germs of the coordinate functions $x_j$, $j = 1,...,n$. For a positive integer $s$ the power ${\mathcal M}(n)^s$ 
consists of germs $f$ from ${\mathcal C}(n)$ such that $D^\alpha f(0) = 0$ for every $\vert \alpha \vert < s$; this ideal 
is generated by the monomials $x_1^{\beta_1}...x_n^{\beta_n}$ with $\vert \beta \vert = s$. If $f \in {\mathcal M}(n)^k$, 
we say that the germ $f$ vanishes with order at least $k$. The space $\mathcal J^k(n)$ can be identified with 
the quotient
$$
{\mathcal C}(n)/{\mathcal M}^{k+1} = \rl[x_1,...,x_n]/\langle x_1,...,x_n\rangle^{k+1},
$$
where $\rl[x_1,...,x_n]$ denotes the ring of polynomials in the variables $x_1,...,x_n$.
Set ${\mathcal M}(n)^\infty = \cap_{s=1}^\infty {\mathcal M}(n)^s$. By Borel's theorem, the quotient 
${\mathcal C}(n)/{\mathcal M}(n)^\infty$ is isomorphic to the ring $\rl[[x_1,...,x_n]]$ of formal power series, i.e., for 
every formal power series there exists a smooth function such that its Taylor series at the origin coincides with this 
power series. In what follows we use the notation $\mathcal J^\infty (n)$ for ${\mathcal C}(n)/{\mathcal M}(n)^\infty$. 
Set also 
$$
\mathcal J^k(n,m) = \underbrace{\mathcal J^k (n) \times ...\times \mathcal J^k (n)}_{m {\rm \ times}}, 
\ \ k=1,2,\dots,\infty .
$$
This is the space of jets of maps from $\rl^n$ to $\rl^m$. Identifying vector fields in $\rl^n$ with maps $\rl^n \to \rl^n$ 
we may also define the jet spaces for vector fields. 
 
\bigskip

Consider $j^k:{\mathcal C}(n,m) \to \mathcal J^k(n,m)$  the map which associates to every germ its $k$-jet at 
the origin. The natural projections
\begin{eqnarray*}
\pi^l_k: {\mathcal J}^l(n,m) \to {\mathcal J}^k(n,m)
\end{eqnarray*}
are defined for $l \geq k$ by
$$j^l (f) \mapsto j^k (f)$$
and satisfy $\pi^l_k \circ \pi^m_l = \pi^m_k$ for $m \geq l \geq k$ and $\pi_{ll} = id$.
Similarly, we also have the canonical projection $\pi^\infty_k : {\mathcal J}^\infty(n,m) \to {\mathcal J}^k(n,m)$. 
A subset $V \subset {\mathcal J}^\infty (n,m)$ is called {\it pro-algebraic} if there exist algebraic subsets 
$V_k \subset {\mathcal J}^k(n,m)$ such that 
$$
V = \cap_{k=1}^\infty (\pi_k^\infty)^{-1}(V_k).
$$ 
The $\sup_k \codim V_{\it k}$ is called the codimension of $V$.

For $k<\infty$ the space $\mathcal J^k (n,m)$ can be identified with the Euclidean space $\rl^d$ of easily 
computable dimension $d = d(k,n,m)$, and so $\mathcal J^k (n,m)$ inherits the structure of a topological
space from this identification. The standard Whitney topology on the space of jets $\mathcal J^\infty(n,m)$ is 
defined as follows. If $U\subset \rl^d\cong \mathcal J^k (n,m)$ is open, then the set 
$M(U):=(\pi^\infty_k)^{-1}(U)$ is defined to be open in $\mathcal J^\infty(n,m)$. The collection of sets
$M(U)$ for all integers $k$ and all open subsets of $\mathcal J^k (n,m)$ is then the basis of the Whitney topology
of $\mathcal J^\infty(n,m)$. In this topology, all projections $\pi^l_k$ and $\pi^\infty_k$ are continuous maps.
Further details can be found in~\cite{gg}.

\subsection{The set $\mathbf A$, the \L ojasiewicz inequality and multiplicity}\label{s.A}
In the proof of Theorem~\ref{t:ga} an important role is played by the pro-algebraic set $A \subset \mathcal J^\infty (2,2)$,
which is defined as follows. Given a vector field $X = X_1(t,s)\frac{\partial}{\partial t}+ X_2(t,s)\frac{\partial}{\partial s}$ in a neighbourhood of the origin in $\rl^2_{(t,s)}$, its jet $j^\infty (X)$ at the origin is not in $A$ if the ideal $\langle X_1,X_2 \rangle$ generated by $X_1$ and $X_2$ contains a power of the maximal ideal $\mathcal M(2)$ in $\mathcal C(2)$. For a germ of a vector field
$X=(X_1,X_2)$ at zero in $\rl^2$ define
\begin{equation}\label{e.tau}
\tau_k(X) =\dim (\mathcal C(2)/(\langle X_1,X_2\rangle+\mathcal M(2)^{k+1}) ,
\end{equation}
and let
\begin{equation}\label{e.Ak}
A_k = \{T_k \in \mathcal J^k(2,2) : \tau_k(T_k) > k-1\} .
\end{equation}
Here we identify $T_k$ with a smooth $X$ such that $j^k(X) = T_k$. For each $k$, $A_k$ is a closed algebraic
subset of $\mathcal J^k(2,2)$, and one can show that
\begin{equation}\label{e.proA}
A = \cap_{k=1}^\infty (\pi_k^\infty)^{-1}(A_k).
\end{equation}
Thus, $A$ is a pro-algebraic set in $\mathcal J^\infty(2,2)$. Further, $\codim A =\infty$. Note that for a given 
germ $X$, its jet $j^\infty(X)$ does not belong to $A$ if and only if $j^k(X) \notin A_k$ for some $k$. This in its turn is equivalent to the fact that the {\it algebraic multplicity} 
$\mu_0(X) := \dim \mathcal C(2)/\langle X_1,X_2\rangle$ is finite. For more 
details about the set $A$  and proofs see \cite{d}.

It is easy to see that if $X=(X_1(t,s), X_2(t,s))$ is a germ of a smooth vector field in $\rl^2$ and $j^\infty(X) \notin A$, then
there exist $k, c, \delta >0$ such that
$$
||X(x)|| \ge c (|t|^2+|s|^2)^k, \ \ |t|^2+|s|^2<\delta .
$$
This is the so-called \L ojasiewicz inequality. \L ojasiewicz proved (see, e.g. \cite{BM}) that the inequality holds
for all  real analytic germs with an isolated zero at the origin. However, if a germ of a real analytic vector field 
has a nonisolated singularity at the origin, it still can be of infinite multiplicity, i.e., its jet can be in $A$. 
Let $X$ be the germ of a real analytic vector field vanishing at the origin. Denote by $X_{\cx}$ the germ of a complex 
analytic vector field defined by the  power series with the same coefficients as $X$, but over a neighbourhood $U$ of 
the origin in $\cx^2$.  Then $X$ is of finite algebraic multiplicity, i.e., $X$ is not in $A$ if and only if  the zero is an 
isolated singularity for the vector field $X_{\cx}$ in $U$. This provides a convenient way to check if a specific real analytic 
germ $X$ is of finite multiplicity. We consider a particular  example of the vector field (\ref{e:hyp-s}) since it will 
be used in the proof of Theorem~\ref{t:ga}.

\begin{ex} \label{Ex0}
Let ${\mathcal X}$ be the vector field (\ref{e:hyp-s}) to be considered later. Then ${\mathcal X}$ is not contained 
in~$A$. Complexifying the objects under consideration, consider the holomorphic polynomial map 
$$f:\cx^2 \to \cx^2 ,$$
$$f = (f_1,f_2):(t,s) \mapsto (-3t^3-ts^2-3t^5,s^3+4t^2 s + 7 st^4) .$$
It is easy to check that there exists a neighbourhood $U$ of the origin  in $\cx^2$ such that 
$f^{-1}(0) \cap U = \{ 0 \}$, i.e., the map $f$ is of finite multiplicity at the origin.
Therefore, $\mathcal X$ is not contained in $A$.
\end{ex}

\subsection{Parametrized complex curves.}
Let $\Omega$ be a bounded domain in $\cx$ whose boundary consists of a finite number of smooth curves $C_j$. 
By a {\it parametrized complex curve} we mean a holomorphic map $h: \Omega \to \cx^n$. If in addition
$h \in C(\overline{\Omega})$, then the restriction $h\vert_{\partial \Omega}$ is 
called the boundary of the curve $h$; we use the same terminology for the image of $h$. Thus the boundary 
$\partial h(\Omega)$ is the union of the curves $h(C_j)$ with the orientation induced by $h(\Omega)$.

If $\Omega = \mathbb D$, the unit disc in $\cx$, then $h: \mathbb D \to \cx^n$ is call a {\it holomorphic}
or {\it analytic disc}, and if $\Omega = A(r,R) =  \{\zeta \in \cx: r < \vert \zeta \vert < R\}$, an annulus of radii
$0<r<R<\infty$, then $h: A(r,R) \to \cx^n$ is called a {\it holomorphic} or {\it analytic annulus}. Its boundary
is the union of two curves $h(C(r))$ and $h(C(R))$, where $C(t)=\{|\zeta|=t\}$. 

As usual, denote by  $\omega = \sum_j dx_j \wedge dy_j$ the standard symplectic form in $\cx^n$. The {\it area} of a parametrized complex curve
$h:\Omega \to \cx^n$ is given by 
$${area\,}(h) = \int_\Omega h^*\omega.$$

For a parametrized complex curve $h$ continuous up to the boundary we say that $h$ is {\it attached} or 
{\it glued} to a set $K$, if $\partial h(\Omega)\subset K$. In Section~\ref{s.5} we will consider a more
refined version of this notion.

\subsection{Currents.}
We briefly recall some standard terminology concerning currents. For a detailed exposition see, e.g., \cite{Dem}. 
Let $\Omega$ be an open subset in $\cx^n$. As usual we set $d^c = i(\overline\partial - \partial)$. Denote by 
${\mathcal D}^{p,q}(\Omega)$ the space of smooth differential forms of bidegree $(p,q)$ with compact support 
in $\Omega$. Its dual space is called {\it the space of currents} of bidimension $(p,q)$ (or bidegree $(n-p,n-q)$). 
A current $S$ of bidimension $(p,p)$, or simply a $(p,p)$-current, is called {\it positive} if for all 
$\a_j \in {\mathcal D}^{1,0}(\Omega)$, $j=1,...,p$, the current 
$S \wedge i\alpha_1 \wedge \overline\alpha_1 \wedge ...\wedge i\alpha_p \wedge \overline\alpha_p$ is a positive 
distribution. The operators $d$ and $d^c$ are defined for currents by duality. A fundamental example of a positive 
$(p,p)$-current is the current $[X]$ of integration over a complex purely $p$-dimensional analytic set $X$ of 
$\Omega$. If $X$ is closed in $\Omega$, then this current is also {\it closed} in $\Omega$, i.e., $d[X] = 0$. 
A current of bidegree $(p,q)$ can be viewed as a differential form of type $(p,q)$ with distributional coefficients. 
If a $(p,p)$-current is positive, then its coefficients are Radon measures (this follows essentially by the Riesz duality 
theorem). The mass $\vert S \vert$ of a positive $(p,p)$-current $S$ is defined by 
$\vert S \vert = S \wedge \omega^p$, where $\omega$ is the standard symplectic form. If the mass is finite, we
use the notation $\parallel S \parallel = \langle S,\omega^p \rangle = \langle \vert S \vert,1 \rangle$. The set 
$\supp (S)$, the support of a current $S$, is defined in the standard way.

\section{Polynomial convexity of an open Whitney umbrella: proof of Theorem~\ref{t:ga}}

\subsection{Reduction to a dynamical system}\label{s:1}

As it is observed in \cite{SS}, $\Sigma$ is contained in the real hypersurface 
\begin{eqnarray}
\label{hypersurf}
M = \{ (z,w) \in \cx^2: \rho(z,w) = x^2 - yv^2 + \frac{9}{4}u^2 - y^3 = 0 \} .
\end{eqnarray}
The defining function $\rho$  of $M$ is strictly plurisubharmonic in a neighbourhood of the origin
 where it admits  an isolated critical point. Hence, $M$ is smooth away from the origin and strictly pseudoconvex in  $\B(0,\e) \setminus \{ 0 \}$, for $\e>0$ small.

The crucial role in our approach is played by the so-called {\it characteristic foliation} induced by $M$ on $\Sigma$. 
Let $X$ be  a totally real surface embedded into  a real hypersurface $Y$ in $\cx^2$.  
Define on $X$ a field of lines determined at every  $p \in X $ by  
$L_p = T_p X \cap H_p Y$,  where $H_p Y = T_p Y \cap J(T_p Y)$ denotes the complex tangent line to $Y$ at the point $p$ and $J$ denotes the standard complex structure of~$\cx^2$.
 Integral curves of the line field $L_p$, i.e., curves which are tangent to $L_p$ at each point $p$,  define a foliation on 
$X$ which  is called the characteristic foliation of $X$.

To a given  local symplectomorphism  $\phi$ we associate  a complex linear map $\psi$ in order to simplify the linear part of $\psi \circ \phi$. One can show (see \cite{SS} for details) that the map $\psi$ can be chosen to depend linearly on the differential $D\phi(0)$ 
so that the differential at the origin of 
the composition $\psi \circ \phi$ is given by
\begin{equation}\label{e:differen}
D(\psi\circ\phi)(0)=
\left(\begin{array}{cc}I_2 & 0\\ E & G\end{array}\right),
\end{equation}
where $I_2$ denotes the identity 2$\times$2 matrix, and $G$ is a nondegenerate 2$\times$2 matrix. Let 
$$\Sigma'=\psi\circ \phi(\Sigma),$$ and $$M'= (\psi\circ \phi)(M).$$  
We put  $$\rho' = \rho \circ (\psi \circ \phi)^{-1}.$$
It follows from \eqref{hypersurf} and \eqref{e:differen} that 

\begin{equation}
\label{e:ro}
\rho'(z',w') = {x'}^2 + \frac{9}{4}{u'}^2 + o(\vert (z',w') \vert^2).
\end{equation} 
In particular,  the function $\rho'$ is 
strictly plurisubharmonic in a neighbourhood of the origin, and the hypersurface $M'$ is strictly pseudoconvex in 
a punctured neighbourhood of the origin.

We consider the characteristic foliation of $\Sigma\setminus\{0\} \subset M$ and 
$\Sigma'\setminus\{0\} \subset (\psi\circ\phi)(M)$.  
A characteristic foliation is invariant under biholomorphic maps. Therefore, in order to study the characteristic foliation on 
$\phi(\Sigma)$ with respect to $\phi(M)$, it is sufficient to study the characteristic foliation of 
$\Sigma' = \psi \circ \phi (\Sigma)$ induced by $M'$.

The main result which we establish in this section is  the following

\begin{prop}\label{l:2}
Let $\phi$ satisfy the assumptions of Theorem \ref{t:ga}. There exist $\e > 0$ small enough and two rectifiable arcs $\gamma_1$ and $\gamma_2$ in  
$\Sigma' \cap \mathbb B (0,\e)$ passing through the origin with the following properties:
 \begin{itemize}
 \item[(i)]  $\gamma_j$ are smooth at all points except, possibly, the origin;
 \item[(ii)] $\gamma_1 \cap \gamma_2 = \{ 0 \}$;  
 \item[(iii)] if $K$ is a  compact subset of $\Sigma' \cap \mathbb B (0,\e)$ and is not contained in
$\gamma_1 \cup \gamma_2$, then there exists a leaf $\gamma$ of the characteristic 
foliation on $\Sigma'$ such that  $K\cap \gamma \ne \varnothing$ but $K$ does 
not meet both sides of $\gamma$.
\end{itemize}
\end{prop}
We point out that by (i) and (ii)  the union $\gamma_1 \cup \gamma_2$ does not bound any subdomain with the closure 
compactly contained in $\Sigma' \cap \mathbb B(0,\e)$.

Once Proposition \ref{l:2} is established, Theorem \ref{t:ga} follows immediately by the method of \cite{SS} which 
does not require real analyticity.  In fact, the only place in \cite{SS} where the real analyticity assumption was used is 
the proof of Proposition \ref{l:2}. The remaining part of this section is therefore devoted to the proof of 
Proposition~\ref{l:2}.

\subsection{Jets and vector fields.} It is shown in \cite{SS} that the pull-back by $\pi$ of the characteristic 
foliation on $\Sigma$ is determined by the system of ODE's of the form

\begin{equation}\label{e:hyp-s}
\left\{
\begin{array}{l}
\dot t = -3t^3-ts^2-3t^5\\
\dot s = s^3+4t^2 s + 7 st^4, 
\end{array}\right.
\end{equation}
where the dot denotes the derivative with respect to the time variable. Similarly, the pull-back of the characteristic 
foliation on $\Sigma'$ is given by
\begin{equation}\label{e:system}
\left\{
\begin{array}{l}
\dot t = \alpha(t,s) = -2g_{12}ts + \alpha_{02}s^2 - 3g_{22}t^3 + o(|t|^3+|s|^2+|ts|)\\
\dot s = \beta(t,s) = 4g_{11}t^2s + \beta_{12}ts^2 + \beta_{03}s^3+ 6g_{12}t^4 + 
o(|t^2s| + |ts^2| + |s|^3+ |t|^4) .
\end{array}
\right.
\end{equation}

System~\eqref{e:system} corresponds to a vector field 
$\mathcal X = \a\frac{\partial}{\partial t}+ \b\frac{\partial}{\partial s}$ defined in a neighbourhood
of the origin in $\rl^2$. As shown in~\cite{SS}  the vector field $\mathcal X$ does not vanish outside the origin,
i.e., the origin is its isolated singularity. We briefly recall from \cite{SS} the construction of the vector 
field $\mathcal X$ from the map $\psi \circ \phi$.
 
Consider the map $f: \rl^2 \to \rl^4$ given by $f:= \psi \circ \phi \circ \pi$. Set $X'_t = \partial f/\partial t$ and 
$X'_s = \partial f/\partial s$. Denote by $\nabla \rho'$ the gradient of the function 
$\rho' = \rho \circ (\psi \circ \phi)^{-1}$. On $\Sigma'$ it can be expressed in terms of $(t,s)$ using the 
parametrization~$f$, namely,
$$
\nabla \rho' = \left(\frac{\partial \rho'}{\partial x'}\circ f, \frac{\partial \rho'}{\partial u'}\circ f,
\frac{\partial \rho'}{\partial y'}\circ f, \frac{\partial \rho'}{\partial v'}\circ f\right).
$$
Then 
\begin{equation}\label{e:ab}
\alpha(t,s) = \langle JX'_s, \nabla \rho' \rangle, \ \ \beta(t,s)=-\langle JX_t, \nabla \rho' \rangle.
\end{equation}

Since the right-hand side of (\ref{e:ab}) is uniquely determined by $\phi$, the map 
\begin{eqnarray}
\label{FundCor}
\Xi: \phi \mapsto {\mathcal X}_\phi = \a\frac{\partial}{\partial t}+ \b\frac{\partial}{\partial s}
\end{eqnarray}
is well-defined; it associates every $C^\infty$ smooth map $\phi$ symplectic at $0$ with the 
vector field $\mathcal X_\phi$. Furthermore, we may extend the definition of $\Xi$ to all smooth
diffeomorphisms defined in a neighbourhood of the origin in $\rl^4$ by the same formula. In
this case the corresponding $\mathcal X$ may vanish also outside the origin.

\begin{rem}
\label{Loj}
In \cite{SS} Proposition \ref{l:2} is established for a $C^\infty$ vector field $\mathcal X_\phi$  under the assumption that it satisfies some generic condition and  the 
\L ojasiewicz inequality. In the real analytic case the \L ojasiewicz inequality holds automatically since the origin  is the only point where $\mathcal X_\phi$ vanishes. Our goal here is to prove that for a generic smooth symplectomorphism $\phi$ the jet of $\mathcal X_\phi$ does not belong to the pro-algebraic set $A$. This   implies the \L ojasiewicz inequality and hence is sufficient in order to deduce  Proposition \ref{l:2} from the results of \cite{SS}.
\end{rem}

Recall that for $k$ positive integer or $\infty$,  $\mathcal J^k (n,m)$ denotes the space of $k$-jets at the origin 
of smooth  maps from $\rl^n$ to $\rl^m$, while the $k$-jet at the origin of a specific map $g$ is denoted by $j^k g$. 
We also consider the subspace $\mathcal J^k_{*}(n,n)$ of 
$\mathcal J^k(n,n)$ consisting of jets of diffeomorphisms at the origin. For each $k\ge 1$, the space of 
jets in $\mathcal J^k(n,n)$ that are not diffeomorphisms is determined by a polynomial equation in $\rl^d$ 
corresponding to $\mathcal J^k(n,n)$, and therefore the space $\mathcal J^k_{*}(n,n)$
is the complement of a codimension one algebraic subvariety of $\mathcal J^k(n,n)$. Thus, $\mathcal J^\infty_{*}(n,n)$ 
is the complement of a pro-algebraic subset of $\mathcal J^\infty(n,n)$ of codimension one.

\begin{lemma}\label{l.Phik}
For every integer $k\ge 1$, the map $\Xi$ induces the map
\begin{eqnarray}\label{l:FundCorJet}
\Xi^{(k)} :{\mathcal J}^{k+1}_{*} (4,4) \to {\mathcal J}^k (2,2)
\end{eqnarray}
defined by
\begin{eqnarray}\label{FundCorJet1}
\Xi^{(k)} :j^{k+1} \phi \mapsto j^k {\mathcal X}_\phi.
\end{eqnarray}
on the space of $k+1$-jets. This map is rational (after the identification of the jet space with the corresponding 
$\rl^d$), with nonvanishing denominator for every map $\phi$. Furthermore, the following diagram commutes
\begin{eqnarray}
\label{diag0}
\begin{CD}
\{\phi : \rl^4 \to \rl^4 \} @>\Xi>> \{\mathcal X: \rl^2 \to \rl^2\}\\
@VVj^{k+1}V	@VVj^k V\\
\mathcal J^{k+1}_* (4,4) @>\Xi^{(k)}>> \mathcal J^k (2,2)  .
\end{CD}
\end{eqnarray}
\end{lemma} 

\begin{proof} Let $\psi$ be defined as in Section~\ref{s:1}, and set $F=\psi\circ \phi$. 
Denote by $x$ coordinates in $\rl^4$. Differentiating the identity $F^{-1}(F(x)) = x$, we obtain
$DF^{-1}(F(x)) = (DF(x))^{-1}$. By Cramer's rule, the components of $(DF(x))^{-1}$, the inverse matrix to 
$DF(x)$, can be expressed as rational functions of the components of $DF(x)$. Differentiation of this identity 
further using the Chain Rule shows that for any integer $k$, derivatives at $F(0)$ of $F^{-1}$ of order $k$  
are rational functions of the derivatives at the origin of $F$ of order up to $k$, and therefore, are rational
functions of the derivatives of $\phi$ at zero. Note that the denominators in these functions are in fact powers 
of the determinant of $DF(0)$ which does not vanish for invertible $\phi$.

By differentiating at the origin $\a$ and $\b$ given by \eqref{e:ab} up to order $k$ one can express 
each derivative as a rational function of the derivatives of $\phi$ at the origin, which shows that the map $\Xi^{(k)}$ 
is rational. Note that because \eqref{e:ab} involves first order derivatives of $\phi$, one requires derivatives of order
$k+1$ of $\phi$ to determine uniquely the $k$-jet of $\mathcal X$. This makes the diagram commute. 
\end{proof}

\begin{lemma}\label{SymplJet}
For every $k\geq 1$ there exists an algebraic subvariety $S_k$ in $\mathcal J^k(4,4)$ which contains $k$-jets at the 
origin of all symplectomorphisms $\phi$. Furthermore, there exists a pro-algebraic set $S\subset \mathcal J^\infty(4,4)$,
such that $j^\infty \phi \in S$ for all symplectomorphisms $\phi$.
\end{lemma}

\begin{proof}
A local diffeomorphism $\phi$ is symplectic if and only if $\phi^*\omega = \omega$. This is a system of first order 
partial differential equations on components of $\phi$. To write them explicitly, denote by $x = (x_1,...,x_4)$ the 
coordinates in $\rl^4$, so $\omega = dx_1 \wedge dx_2 + dx_3 \wedge dx_4$ and $\phi = (\phi_1,...,\phi_4)$. 
Denote by ${\rm Jac}(l,m,j,k)$ the determinant of the 2$\times$2 minor of $D\phi(x)$ corresponding to derivatives of 
$(\phi_l,\phi_m)$ with respect to variables $(x_j,x_k)$, i.e.,
$$
{\rm Jac}(l,m,j,k) = \frac{\partial \phi_l}{\partial x_j}\frac{\partial \phi_m}{\partial x_k} - 
\frac{\partial \phi_m}{\partial x_j}\frac{\partial \phi_l}{\partial x_k} .
$$
Consider the system of equation 
\begin{eqnarray}
\label{SympPDE}
{\rm Jac}(1,2,j,k) + {\rm Jac}(3,4,j,k)  - d_{jk} = 0, \ \ j < k,
\end{eqnarray}
where $d_{jk}$ is equal to  $1$ if $j=1$, $k = 2$ or $j=3$, $k = 4$ and to $0$ in other cases. For $x=0$
this system can be interpreted as a polynomial equation in $\mathcal J^k (4,4)$, say, if $j^1\phi$ satisfies this
equation then it simply means that $D\phi(0)$ is symplectic. Applying to (\ref{SympPDE})  the partial 
derivative operators up to the order $k-1$ we obtain that $j^k\phi$ satisfies a polynomial system 
of equations in $\rl^d$ corresponding to $\mathcal J^k(4,4)$.  Finally, the set
\begin{equation}\label{e.sgerms}
S= \bigcap_{k=1}^\infty (\pi^\infty_k)^{-1} (S_k)
\end{equation}
is pro-algebraic and contains jets of all symplectic maps.
\end{proof}

\subsection{H\'enon-like symplectic maps and polynomial approximation.}\label{s.turaev} 
In the proof of Proposition~\ref{l:2} we will need the  results on (real) symplectic polynomial approximation due to Turaev~\cite{Tur}, which we 
describe below. Note that the problem of symplectic polynomial approximation of holomorphic symplectomorphisms was studied by Forstneri\v c \cite{Fo1,Fo2}. For the standard symplectic form $\omega=\sum_j dx_j \wedge dy_j$ in $\rl^{2n}$ with
coordinates $(x,y)=(x_1, \dots, x_n, y_1, \dots, y_n)$, a {\it H\'enon-like map} is a symplectic map given by 
\begin{equation}\label{e.henon}
H(x,y) \to (y, -x + \nabla V(y)),
\end{equation}
where $V: \rl^n \to \rl$ is a smooth function. A H\'enon-like map is a global diffeomorphism of $\rl^{2n}$ and the 
inverse is also a H\'enon-like map. If $V$ is a polynomial, then $H$ is a polynomial map as well and is called 
{\it a polynomial H\'enon-like map}. Given a smooth symplectic diffeomorphism $\phi: \mathbb B_{n} \to \rl^{2n}$, 
a compact  $K \subset \mathbb B_{n}$, an integer $m>0$, and $\e>0$,  there exists a collection of symplectic polynomial 
H\'enon-like maps $\{H_j\}$, $j=1,\dots, N$, such that the composition $H_N \circ \dots \circ H_1$ approximates $\phi$ 
with accuracy $\e$ in the $C^m$-topology on $K$. By taking $K$ to be the origin, Turaev's theorem
gives for any $m>0$ a symplectic polynomial map whose jet of order $m$ at the origin is arbitrarily 
close to that of a given smooth symplectic map. In particular, this means that jets of symplectic polynomials 
which are compositions of polynomial maps of the form \eqref{e.henon} are dense in the Whitney topology in the 
space of all jets of symplectic maps in $\mathcal J^\infty (2n, 2n)$. 

Turaev \cite[Thm 1]{Tur} proved also  a more precise result allowing the choice of the approximating symplectic polynomial map 
$H_N$ in a more restricted form:

\begin{theorem}
\label{Turaev}
 Let $U$ be a ball in $\mathbb R^{2n}$, and let $\phi: U \to \mathbb R^{2n}$ be a 
$C^r$-smooth (with an integer $r > 0$) symplectic diffeomorphism. Then, for any compact set $K\subset U$ and any $\e>0$ there exists
a polynomial $V: \mathbb R^n \to \mathbb R$, a constant vector $\eta\in \mathbb R^n$ and an integer $N>0$
such that the $4N$-th iteration 
\begin{eqnarray}
\label{Henon1}
H_N = \underbrace{H^\eta \circ ... \circ H^\eta}_{4N\ {\rm times}}
\end{eqnarray} 
of the symplectic map 
\begin{eqnarray}
\label{Henon2}
H^\eta: (x,y) \mapsto (y + \eta,  -x + \nabla V(y)) 
\end{eqnarray}
approximates $F$ with the accuracy $\e$ in the $C^r$-topology:
\begin{eqnarray*}
\parallel \phi - H_N \parallel_{C^r(K)} < \e .
\end{eqnarray*}
\end{theorem}

Notice that $H^\eta$ is not a H\'enon-like map for $\eta \neq 0$. However, it coincides with a H\'enon-like map 
after the shift of coordinates $y \mapsto y + \eta$. 

\begin{rem}
\label{rem1}
For the Henon-like map $H^0:(x,y) \mapsto (y,-x)$ (corresponding to $\eta = 0$ and $V = 0$) its  4-th iteration is equal to the identity map. This explains why the map $H_N$ is defined by means of the $4N$-th iteration of $H^\eta$.
\end{rem}

From now on we restrict our considerations to the case of  symplectomorphisms defined on  $\rl^4$, i.e., on $\cx^2$ with the coordinates $z_j = x_j + iy_j$, $j=1, 2$. Denote by $\rl[y_1,y_2]$ (resp. $\rl_l[y_1,y_2]$) the vector space of real polynomials (resp. of degree $\le l$) in $y_1$, $y_2$. Given $N>0$, any  polynomial $V \in \rl[y_1, y_2]$  and a vector $\eta \in \rl^2$ define an $H$-map  by formulas (\ref{Henon1}), (\ref{Henon2}). Denote this map by $H_{N,V,\eta}$ and the set of all such maps by $\mathcal H(N)$. In what follows we call in short these symplectic polynomial maps {\it H-maps of order $N$}.  Of course, in general the degree of a polynomial $H_N$ can be  higher than $N$. If in the above definition we consider only 
$H$-maps corresponding to generators $V \in \rl_l[y_1,y_2]$, we obtain a subset of $\mathcal H(N)$ denoted by 
$\mathcal H(l,N)$. The degrees of these maps are bounded above by $l^{4N}$ (by $(l-1)^{4N}$ for $l > 1$). Clearly,
$\mathcal H(l,N) \subset \mathcal H(l+1,N)$ and 
$$\mathcal H(N) = \cup_{l\geq 0} \mathcal H(l,N) .$$ 
Denote by $\mathcal V(l,N)$ the vector space  of polynomial maps $\rl^4 \to \rl^4$ of degree at most $l^{4N}$. 
We obtain a well-defined map
$$\theta: \rl_l\,[y_1,y_2] \times \rl^2 \to \mathcal V(l,N) ,$$
$$\theta: (V,\eta) \mapsto H_{N,V,\eta} .$$

Denote by $i_1:\rl_l[y_1,y_2] \to \rl^d$, $d = d(l)$, the standard linear isomorphism that associates with a polynomial $P$  
the vector of $\rl^d$ formed by the coefficients of $P$. Similarly, we also identify the vector space $\mathcal V(l,N)$  
with some $\rl^p$, $p = p(l,N)$ by means of an isomorphism $i_2$. We obtain the  map
$$\Theta: \rl^d \times \rl^2 \to \rl^p$$
defined by the commutative diagram
\begin{eqnarray}
\label{diag1}
\begin{CD}
\rl_l[y_1,y_2] \times \rl^2  @>\theta>> \mathcal H(l,N) \subset\mathcal V(l,N)\\
@VVi_1 \otimes id V	@VV i_2 V\\
\mathcal \rl^d \times \rl^2 @>\Theta>> \rl^p  .
\end{CD}
\end{eqnarray}

\smallskip

Clearly, the map $\Theta$ is polynomial.  One can view the space $\rl^d \times \rl^2$ as the moduli space for the 
set  $\mathcal H(l,N)$.  The inverse map 
$$(H^{\eta})^{-1}: (x',y') \mapsto (-y' + \nabla V(x'-\eta), x' - \eta)$$ 
also is a polynomial map of the same degree as $H^\eta$ with coefficients depending polynomially on the coefficients of 
the generating polynomial $V$ and the vector  $\eta$. Hence, the maps 
$$\tilde\theta: \rl_l\,[y_1,y_2] \times \rl^2 \to \mathcal V(l,N),$$
$$\tilde \theta: (V,\eta) \mapsto (H_{N,V,\eta})^{-1},$$
also define  a polynomial map
$$\tilde\Theta:\rl^d \times \rl^2 \to \rl^p$$ 
by the commutative  diagram
\begin{eqnarray}
\label{diag2}
\begin{CD}
\rl_l[y_1,y_2] \times \rl^2  @>\tilde\theta>> \mathcal H(l,N)\subset \mathcal V(l,N)\\
@VVi_1 \otimes id V	@VV i_2 V\\
\mathcal \rl^d \times \rl^2 @>\tilde\Theta>> \rl^p  .
\end{CD}
\end{eqnarray}

\smallskip

For $H_{N,V,\eta} \in \mathcal H(l,N)$ consider the vector field $\mathcal X_{H_{N,V,\eta}} = \Xi(H_{N,V,\eta})$ defined by (\ref{FundCor}). Using the described above global  parametrizations of $\mathcal H(l,N)$ by its moduli space via $\Theta$ and $\tilde\Theta$, and the definition of the map $\Xi$ in (\ref{FundCor}), we obtain  that $\mathcal X_{H_{N,V,\eta}}$ is a polynomial vector field.  Its degree  (the maximal degree of its coefficients) is bounded by some $d' = d'(l,N)$.
Furthermore, the coefficients of $\mathcal X_{H_{N,V,\eta}}$ are polynomial functions of the coefficients of $V$ and 
vector $\eta$. By analogy with the above construction, let us   identify the space $\mathcal W(d')$ of polynomial vector 
fields of degree at most $d'$ with some $\rl^q$, $q = q(l,N)$, by means of an isomorphism 
$i_3: \mathcal W(d') \to \rl^q$. We obtain the following

\begin{lemma} \label{HenLem1}
The map 
$$\Lambda:\rl^d \times \rl^2 \to \rl^q$$
defined by the commutative diagram
\begin{eqnarray}
\label{diag3}
\begin{CD}
\mathcal \rl^d \times \rl^2  @>\Lambda>> \rl^q \\
@VVi_2^{-1} \circ \Theta V	@VV i_3^{-1} V\\
\mathcal H(l,N) @>\Xi>>  \mathcal W(d')
\end{CD}
\end{eqnarray}
\bigskip
is a polynomial map.
\end{lemma} 
This lemma implies that the maps $\Xi^{(k)}$ restricted to the set of jets of symplectomorphisms from $\mathcal H(l,N)$ 
are not just rational, but even polynomial. 

\subsection{Proof of Proposition \ref{l:2}} 
The pro-algebraic set $A\subset \mathcal J^\infty(2,2)$ defined in (\ref{e.proA}) plays a central role in the local 
theory of vector fields in $\rl^2$ with isolated singularities in view of the work of Dumortier \cite{d}. 
In order to state its result, we still need to recall some standard notions from the local theory of dynamical systems.

Two germs $\tilde X$, $\tilde Y$ of vector fields at the origin in $\rl^n$ are called {\it topologically equivalent} 
(or $C^0$-equivalent) if for some (and hence for all) representatives $X$, $Y$, there exist neighbourhoods $U$ 
and $W$ of the origin in $\rl^n$, and a homeomorphism $h:U \longrightarrow W$,  mapping integral curves of 
$X$ to integral curves of $Y$ preserving the phase portrait, but not necessarily the parametrization.

The jet $j^k X$ of a vector field $X$ is called $C^0$-{\it determining} if  any germ $Y$ with 
$j^k (X) = j^k (Y)$ is topologically equivalent to $X$. We say that (a germ of) a vector field $X$ is {\it finitely 
determined} if there exists some $k$ such that $j^k(X)$ is $C^0$-determining. A {\it characteristic orbit}
of $X$ is an integral curve that asymptotically approaches the singularity in such a way that the tangent
line has a well-defined limit. The main theorem of \cite{d} (see also Theorem 3.3 in \cite{dull} or Theorem 2.2  \cite{duro} for a more detailed account and more general results) states that if $X$ has a characteristic orbit and 
$j^\infty(X)$ is not contained in $A$ (in fact, it suffices to impose a weaker assumption that $j^\infty(X)$ satisfies  the \L ojasiewicz inequality), then $X$ is finitely determined. Further, whether $X$ has a characteristic orbit depends only on a jet of $X$ of some finite order.

Let $A_k$ be defined as in \eqref{e.Ak}. Since by Lemma~\ref{l.Phik} the map $\Xi^{(k)}$ is rational,
the set
$$
\tilde A = \bigcap_{k=1}^\infty \tilde A_k ,
$$
where $$\tilde A_k = (\pi^\infty_k)^{-1} \left(\Xi^{(k)}\right)^{-1}(A_k),$$
is pro-algebraic in $\mathcal J^\infty (4,4)$. Let $\tilde S \subset \mathcal J^\infty(4,4)$ be the subset of jets of symplectomorphisms, so $\tilde S \subset S$. 
\begin{lemma}
\label{density}
The intersection  $\tilde S \cap \tilde A$ is closed nowhere dense in $\tilde S$ in the Whitney topology.
\end{lemma}
\begin{proof}
The set $\tilde A$ is closed, so 
$\tilde S \cap \tilde A$ is closed in $\tilde S$. Therefore we only need to prove that $\tilde S \cap \tilde A$ is nowhere dense in $\tilde S$.

We will use notations introduced in Section~\ref{s.turaev}. Arguing by contradiction suppose that there exist a symplectomorphism $\phi$ and an open neighborhood  $U$ of  $j^\infty \phi$ in 
$\mathcal J^\infty(4,4)$ such that $\tilde S \cap U$ is contained in $\tilde S \cap \tilde A$. 

From the properties 
of the Whitney topology, there exist $m>0$ and $\e>0$ such that 
$$
(\pi^\infty_m)^{-1}\left(\mathbb B(j^m\phi, 2\e) \right) \subset U,
$$
where the ball is in $\rl^s$ corresponding to $\mathcal J^m(4,4)$. By Theorem \ref{Turaev}, the map $\phi$ 
can be approximated by polynomial symplectomorphisms in the $C^m$-topology, and so
we conclude that there exist $l,N$ and  a  polynomial symplectomorphism $H \in \mathcal H(l,N)$ such that  
$j^m H \in \mathbb B(j^m\phi, \e)$. This means that $j^\infty H\in \tilde S \cap U$,  so by our assumption, $j^\infty H\in \tilde S \cap \tilde A$.

For $k$ big enough, every polynomial symplectomorphism from $\mathcal H(l,N)$ is uniquely determined by its $k$-jet. So we can assume that $\mathcal H(l,N)$ is included in $J^k(4,4)$. If $k$ satisfies  additionaly the condition $k \geq max(m,d')$ (where $d' = d'(l,N)$ is defined in Section~\ref{s.turaev}), then every polynomial vector field from $\mathcal W(d')$ is
also  uniquely determined by its $k$-jet, and we can view $\mathcal W(d')$ as a subset of the jet space $\mathcal J^k(2,2)$ as well.
 Fix now such a $k$.  The map $\Lambda$ of Lemma~\ref{HenLem1}
defines a polynomial map
\begin{equation}\label{e.phi}
\Phi : \rl^d \times \rl^2 \to J^k(2,2)
\end{equation}
given by $\Phi = j^k \circ i_3^{-1} \circ \Lambda $ and satisfying
$j^k \circ \Xi (i_2^{-1} \circ \Theta (\xi)) = \Phi(\xi)$ for all $\xi \in  \rl^d \times \rl^2$.

 Since $A_k$ is an algebraic subset of $J^k(2,2)$ and the map $\Phi$ is polynomial, the pull-back   $Z:=\Phi^{-1}(A_k)$ is a proper real algebraic subvariety of $\rl^{d+2}$ defined by 
 $$
 Z = \{ x \in \rl^{d+2}: Q_j(x) = 0, \ j=1,...,s\}.
 $$ 
 Here $Q_j:\rl^{d+2} \longrightarrow \rl$ are polynomials. Recall that by  Remark \ref{rem1} the image $i_2^{-1} \circ \Theta(0)$ of the origin $0 \in \rl^{d+2}$ is the jet of the identity map in $\mathcal H(l,N)$ (cf.  diagram ~\eqref{diag1}). It follows from diagram ~\eqref{diag3}  that the image $\Phi(0)$  coincides with the jet of the vector field ~\eqref{e:hyp-s}.
Therefore, as seen in Example~\ref{Ex0},  the origin $0 \in \rl^{d+2}$ does not belong to $Z$;  in particular, $Z$ is a proper subset of $\rl^{d+2}$.

Since $j^\infty H \in \tilde A$, the point $p: = (\Theta^{-1}\circ i_2)(H)$ belongs to  $Z$ (cf. diagram (\ref{diag1})). Consider in $\rl^{d+2}$ the real line $L: \rl \ni t \mapsto tp$   through the origin and $p$. Since the line $L$ is not contained in the real algebraic set $Z$, at least one of the polynomials $Q_j(tp)$ does not vanish identically in $t \in \rl$; hence $L$  intersects $Z$ in a finite set of points. Choosing a point $\tilde p \in L\setminus Z$ close enough to $p$, we obtain the  
 polynomial symplectomorphism $\tilde H:= (i_2^{-1} \circ \Theta)(\tilde p) \in \mathcal H(l,N)$ satisfying $j^m\tilde H \in \B(j^mH,\e)$.  Then for the corresponding vector field $\mathcal X_{\tilde H} = \Xi(\tilde H) \in \mathcal W(d')$ we have $j^\infty(\mathcal X_{\tilde H})\notin A$, and so $j^\infty \tilde H\notin \tilde A$ -- contradiction. \end{proof}

Now we are able to conclude the proof of Proposition \ref{l:2}. By Lemma \ref{density},
for a smooth symplectomorphism $\phi$  the condition $j^\infty \phi \notin \tilde A$ is
generic, i.e., it holds for a subset  ${\mathcal O}_1$, which open and dense in $\tilde S$. 
For such a symplectomorphism  the corresponding vector field $\mathcal X_\phi$ defined by (\ref{FundCor}) satisfies $j^\infty\mathcal X_\phi \notin A$ and, in particular, satisfies the \L ojasiewicz inequality.
Denote by ${\mathcal O}_2$ the set of jets $j^\infty\phi \in \tilde S$ such that   the  coefficients $\alpha_{02}$, $\beta_{12}$, $\beta_{03}$ 
in \eqref{e:system} (i.e., from the Taylor expansion of the vector field $\mathcal X_\phi$) do not vanish. As shown in \cite{SS}, this condition is also generic, i.e., ${\mathcal O}_2$ is an open dense subset of $\tilde S$. Hence,  
the intersection ${\mathcal O} := {\mathcal O}_1 \cap {\mathcal O}_2$ also is open dense in $\tilde S$.

Consider  a symplectomorphism $\phi$ with $j^\infty \phi\in \mathcal O$. We need to show that the corresponding 
vector field $\mathcal X_\phi$ satisfies Proposition \ref{l:2}. 
Since $j^\infty \phi\in \mathcal O$, the \L ojasiewicz inequality for $\mathcal X_\phi$ holds, and
in view of the results in~\cite{SS} this already implies Proposition~\ref{l:2} (cf. Remark \ref{Loj}).
For reader's convenience we summarize the argument. Recall that whether a vector field satisfying the 
\L ojasiewicz inequality has a characteristic
orbit is determined by its jet at the origin of some finite order. Using truncation of the Taylor expansion of 
$\mathcal X_\phi$ of arbitrarily high degree and its Newton diagram, it was shown in~\cite{SS} that 
the topological phase portrait of the truncated vector field~\eqref{e:system} is a ``saddle", i.e., it 
precisely satisfies requirements (i), (ii), and (iii) of Proposition \ref{l:2}. In particular, it has characteristic
orbits, the curves $\gamma_j$ of Proposition \ref{l:2}. Hence, by the theorem of Dumortier,
the vector field $\mathcal X_\phi$ is finitely determined, and therefore, $\mathcal X_\phi$ 
has the same phase portrait as its polynomial truncation of some sufficiently high degree, that is the jet at zero of
$\mathcal X_\phi$ of some finite order. This concludes the proof of Proposition \ref{l:2} and proves Theorem \ref{t:ga}.

\medskip

In what follows by a {\it generic} open Whitney umbrella point, or simply a generic umbrella point, we mean a point 
for which the conclusion of Theorem~\ref{t:ga} holds. If  $\phi$ is an arbitrary (not necessarily generic)  symplectomorphism, we call $\phi(0)$ {\it an open umbrella point} for the surface $\phi(\Sigma)$.

\begin{rem}
Instead of considering the space $\tilde S$ of jets of symplectomorphisms as above, we may consider the space 
of jets of maps which only have a symplectic linear part at zero. Such maps define characteristic foliation on 
$\Sigma'$ which can be singular at more points than just the origin. However, a similar argument as in 
Theorem~\ref{t:ga} shows that for a generic map $\phi$ symplectic at zero, the characteristic foliation is singular 
only at zero in a small neighbourhood of the origin, and has the phase portrait determined by \eqref{e:system}. 
Thus, Theorem~\ref{t:ga} also holds under a weaker assumption that $D\phi(0)$ is symplectic. 
\end{rem}

\section{Polynomial convexity near double points: proof of Theorem~\ref{InterTheo}}\label{s.X}

We recall the result  due to Weinstock \cite{Wein} that can be stated as follows:

\medskip

{\it Let $E_1$ and $E_2$ be maximally totally real linear subspaces of $\cx^n$ intersecting transversally at the origin. 
Then either the union $E_1 \cup E_2$ is polynomially convex or there exists an  analytic annulus 
$h : A(r,R) \to \cx^n$ smooth up to the boundary such that $h(C(r)) \subset E_1$  and $h(C(R)) \subset E_2$, i.e., 
the union $E_1 \cup E_2$ contains the boundary of $h$.}

\medskip

As a consequence we have 
\begin{lemma}\label{Eigen}
Let $E_1$ and $E_2$ be Lagrangian subspaces of $\cx^n$ intersecting transversally at the origin. Then  the union $E_1 \cup E_2$ is polynomially convex.
\end{lemma} 

\begin{proof} It suffices to show that $E_1 \cup E_2$ does not contain the boundary of an analytic annulus. Arguing by contradiction, suppose that there exists an analytic annulus $h$ attached 
to $E_1 \cup E_2$. 
Let $\lambda$ be a 1-form on $\cx^2$ such that $d\lambda = \omega$. By Stokes' formula 
$$\int_{h(A(r,R))} \omega  = \int_{h(C(r))} \lambda + \int_{h(C(R))} \lambda.$$
The restriction $\lambda\vert_{E_1}$ is a closed form because $E_1$ is Lagrangian. Since the closed curve  $h(C(r))$, 
resp. $h(C(R))$, is null-homotopic  in $E_1$, resp. in  $E_2$,  we conclude that both integrals on the right vanish. 
Therefore, the integral on the left also vanishes. But since the map $h$ 
is holomorphic, this integral represents the area of $h(A(r,R))$ with respect to the usual Euclidean 
metric and so is different from zero. This contradiction shows that the union $E_1 \cup E_2$ does not contain
 a boundary of an analytic annulus. Thus $E_1 \cup E_2$ is polynomially convex by Weinstock's theorem.
\end{proof}
 
The next theorem generalizes Weinstock's result to submanifolds.  
\begin{theorem}\label{t.3}
\label{t:inter2}
Let $L_1$ and $L_2$ be smooth totally real submanifolds in $\cx^n$ intersecting transversally at the origin. 
Suppose that the union of their tangent spaces at the origin is locally polynomially convex near the origin. 
Then the union $(L_ 1 \cup L_2)$ is locally polynomially convex near $0$.
\end{theorem}

Gorai \cite{Go} proved this for $n=2$. Also note that Theorem \ref{InterTheo} is an 
immediate consequence of Theorem~\ref{t:inter2} and Lemma~\ref{Eigen}.

\begin{proof}[Proof of Theorem~\ref{t.3}.] Let $E_j = T_0 L_j$, $j=1, 2$. After a complex linear change
of coordinates we may assume that $E_1 = \rl^n_x$, $x=(x_1, \dots, x_n)$, $z_j=x_j+i y_j$.
Then $E_2$ can be expressed as the graph of a real valued linear map $A: \rl^n_y \to \rl^n_x$,
so the points $z\in E_2$ can be given by $z=(A+iI_n)y$, where $I_n$ is the identity $n\times n$
matrix. Weinstock \cite{Wein} proved that {\it  the union  $E_1 \cup E_2$ is locally polynomially
convex near the origin if and only if $A$ does not have purely imaginary eigenvalues of absolute 
value bigger than one.}

By Lemma \ref{Eigen} the above condition on eigenvalues of $A$ holds if $E_j$ are Lagrangian spaces. 
As in \cite{Wein}, our argument is based on Kallin's lemma (see, for instance, \cite{St}). For readers's 
convenience we recall its statement:

\medskip

{\it Let $X_1$ and $X_2$ be compact, polynomially convex subsets of $\cx^n$. Let $p$ be a polynomial such that 
the set $p^{-1}(0) \cap (X_1 \cup X_2)$ is polynomially convex. Assume that the polynomially convex sets 
$\widehat{p(X_j)}$, $j=1, 2$, of $\cx$ meet only at the origin which is a boundary point for each of them. Then the set 
$X_1 \cup X_2$ is polynomially convex.}

\medskip

We will deal with the special case where $p(X_1)$ and $p(X_2)$ are contained in 
$\{ \xi + i\eta \in \cx \,\vert\, \eta < 0 \} \cup \{ 0 \}$ and $\{ \xi + i\eta \in \cx \,\vert\, \eta > 0 \} \cup \{ 0 \}$  respectively. Note that this implies that
the sets $\widehat{p(X_j)}$ also satisfy this property.

A complex linear change of coordinates of $\cx^n$ defined by $z \mapsto Bz$, where $B$ is a real 
$n \times n$-matrix, will transform $E_2$ into the linear space $z=(BAB^{-1} + iI_n)y$, and so we may assume 
without loss of generality that the matrix $A$ is in the Jordan normal form:
\begin{equation}\label{e.A}
A = \left(\begin{array}{ccccc} 
A_1 & 0 & 0 & \dots & 0\\ 
0 & A_2 & 0 & \dots & 0\\
& & \dots & & \\
0 & 0 & 0 & \dots & A_m
\end{array}\right),
\end{equation}
where each $A_j$ is either a Jordan block of the form 
\begin{equation}\label{e.Ajreal}
A_j =\left(\begin{array}{ccccc} 
\lambda_j & 1 & 0 & \dots & 0\\ 
0 & \lambda_j & 1 & \dots & 0\\
& & \dots & & \\
0 & 0 & 0 & \dots & \lambda_j
\end{array}\right),
\end{equation}
for a real eigenvalue $\lambda_j$ of $A$, or of the form
\begin{eqnarray}\label{block2}
A_j=\left(\begin{array}{ccccc}
C_j & I_2 & 0 & \dots &0\\
0 & C_j & I_2 & 0 & \dots \\
0 & 0 & C_j & I_2 & \dots \\
& & \dots & & \\
0 & 0 & 0 & 0 & C_j \\
\end{array}\right), \ \ 
C_j=\left(\begin{array}{cc} 
s_j & -t_j\\ t_j & s_j
\end{array}\right)
\end{eqnarray}
for a complex eigenvalue $s_j + i t_j$.

Applying  the implicit function theorem to the defining equations of $L_1$, we conclude that for a small 
polydisc $U \subset \cx^n$,
\begin{equation}\label{e.s1}
L_1 = \left\{z \in U : z=x+i\phi(x) \right\},
\end{equation}
where $\phi: \rl^n \to \rl^n$ is a smooth map with $\phi(0)=0$, and $\frac{\partial \phi_j}{\partial x_k}(0)=0$
for $j,k = 1,\dots, n$. Similarly, 
\begin{equation}\label{e.s2}
L_2 = \left\{z \in U : z=(Ay+\psi(y))+iy \right\},
\end{equation}
with $\psi(0)= \frac{\partial \psi_j}{\partial y_k}(0)=0$, $j,k = 1,\dots, n$. When $A$ is in the Jordan normal
form, then the transformation $x=Ay$ can be split into $m$ transformations of the form 
$x_{\mu_j} = A_j \, y_{\mu_j}$, where $x_{\mu_j}$ and $y_{\mu_j}$ are the appropriate subvectors of $x$ and $y$
of size equal to the dimension of the block $A_j$. Because of this decomposition, we may construct polynomials
$p_j(z_{\mu_j})$ corresponding to each block $A_j$, and then combine the results together to obtain a polynomial 
$p(z)$ for $A$ that will satisfy the assumptions of Kallin's lemma. We consider several cases depending on the shape of 
$A_j$. In what follows $C$ denotes a positive constant which may change from line to line.

\bigskip

{\it Case 1.} Let $A_j=(\lambda_j)$. Denote by $z_j = x_j +iy_j$ the corresponding variable. Consider the polynomial
$$p_j(z) = (\lambda_j - i)z_j^2.$$
Then 
$$\Im p_j \vert_{L_1} = -x_j^2 + O(\parallel x \parallel^3),$$
while
$$ \Im p_j \vert_{L_2}= (\lambda_j^2 + 1) y_j^2 + O(\parallel y \parallel^3).$$

\bigskip

{\it Case 2.} Suppose now that
$$
A_j = \left(\begin{array}{cc} 
s_j & -t_j\\ t_j & s_j
\end{array}\right) .
$$
Let $z_j$, $z_{j+1}$ be the variables corresponding to this block. Set 
$$p_j(z) = (s-\delta i)(z_{j}^2 + z_{j+1}^2),\ \ \delta>0 .$$
Then 
$$\Im p_j \vert_{L_1} = -\delta (x_j^2 + x_{j+1}^2) + O(\parallel x \parallel^3),$$
and 
$$\Im p_j \vert_{L_2} = (2s_j^2 - \delta (s_j^2 + t_j^2 - 1))(y_j^2 + y_{j+1}^2) + O(\parallel y \parallel^3).$$
If $s_j \ne 0$, then by choosing $\delta> 0$ small enough we ensure that $\Im p_j \vert_{L_2} \ge C(x_j^2+x_{j+1}^2)$.
If $s_j=0$, then $\vert t_j \vert < 1$ and the same estimate holds with $\delta = 1$.

\bigskip

{\it Case 3.} Suppose that $A_j$ is a Jordan block of size $k$ as in \eqref{e.Ajreal}.
Without loss of generality assume that the block $A_j$ corresponds to the first $k$ coordinates,
i.e., to $z_1,...,z_k$. In what follows we use the convention that $x_{l}=y_{l}=0$ for $l>k$ to 
simplify the formulas with summation. Consider  the polynomial 
$$p_j(z)= \sum_{l=1}^k (\a_l - \delta i) z_l^2,$$
where $\a_l$ and $\delta$  will be suitably chosen positive constants. For any $z\in L_1$ we have 
\begin{eqnarray}
\label{norm1}
\Im p(z) =  - \delta \sum_l x_l^2 + O(\parallel x \parallel^3) .
\end{eqnarray}
Suppose now that $z = (A+iI_n)\,y \in E_2$. Then 
$$
\Im p(z) = \lambda_j \left(\alpha_1 y_1^2+ \alpha_k y_k^2\right)+ \lambda_j
\sum_{l=1}^{k-1} \left(\alpha_l y_l^2+\frac{2 \alpha_l}{\lambda_j} y_l y_{l+1} + 
\alpha_{l+1}y_{l+1}^2 \right) + q_\delta(y) + O(\parallel y \parallel^3). 
$$
Here and below $q_\delta(y)$ denotes a quadratic form in $y_1,\dots,y_k$ of the norm smaller than $\delta$. 
Since a quadratic polynomial $ax^2+bxy+cy^2$ is nonnegative (resp. nonpositive) whenever $a > 0$ 
(resp. $a < 0$) and $b^2 \le 4ac$, we may choose $\alpha_l > 0$  inductively, starting with $\alpha_1=1$, 
so that every term in the sum above is nonnegative. Further, a choice of $\delta$ small enough will ensure
that 
\begin{equation}\label{e.normy}
\Im p(z) \ge  C \sum_{j=1}^k y_j^2 .
\end{equation}

\bigskip

{\it Case 4.} Consider the Jordan block of size $k$ given by~\eqref{block2} corresponding to a complex eigenvalue 
$s_j+it_j$ of $A$. As it was mentioned previously, we always have $|t_j|<1$ if $s_j=0$. This follows from Weinstock's 
criterion of polynomial convexity for the union of two totally real spaces which holds in the Lagrangian case.
Consider the polynomial 
\begin{equation}\label{lastp}
p_j(z) = (s_j-\delta i)(\a_1 z_1^2 + \dots + \a_k z_k^2),
\end{equation}
where $\a_l$ and $\delta$ are some positive constants. Then 
$$
\Im p_j \vert_{L_1} = -\delta (\a_1 x_1^2 + ...+\a_k x_k^2) + O(\parallel x \parallel^3).
$$

As for $L_2$, first note that $k$ is necessarily even, and we set $\a_{2l-1} = \a_{2l}$, $l=1,2,\dots,k/2$.
If $s_j \ne 0$, then for $z\in L_2$ we have
$$
\Im p_j(z) = 2s_j^2\left(\sum_{l=1}^{k} \a_{l}( y_l^2 +  \a_l\, y_l\, y_{l+2})\right) + 
q_\delta(y) + O(\parallel y \parallel^3) ,
$$
where $q_\delta$ has the same properties as above. As in Case 3, we may choose coefficients $\a_l$ inductively
so that the first term on the right-hand side above is positive-definite, and further we may choose $\delta >0$ 
small enough so that
\begin{equation}\label{e.y2}
\Im p_j\vert_{L_2} \geq C (y_j^2 + \dots + y_k^2) .
\end{equation}
If $s_j=0$, then
\begin{eqnarray*}
\Im p_j \vert_{L_2} = \delta \left( \a_1(1-t_j)^2(y_1^2+y_2^2) + 
\sum_{l=2}^{k/2} (\a_{2l-1}(1-t_j^2)-\a_{2l-3}) (y_{2l-1}^2+y_{2l}^2) + \right. \\
\left. 2t \sum_{l=1}^{k/2} \a_{2l-1}(y_{2l-1}y_{2l+2} - y_{2l}y_{2l+1}) \right) .
\end{eqnarray*}
Again, coefficients $\a_l$ can be chosen in such a way that  the required estimate~\eqref{e.y2} holds.
\bigskip

Now to combine all cases together, consider
$$p(z) = \sum_{j=1}^m p_j(z) ,$$
where $p_j$ are the polynomials constructed above for each Jordan block of $A$. Then 
$$
\Im p \vert_{L_1} \leq - C \parallel x \parallel^2, {\rm \  and \ } \Im p \vert_{L_2} \geq C \parallel y \parallel^2.
$$
Note that these estimates are possible precisely because in $p(z)$ constructed above 
{\it all} quadratic terms $z_\nu^2$, $\nu=1, 2, \dots, n$ are present.
Hence $p(L_1)\subset \{ \xi + i\eta \in \cx \,\vert\, \eta < 0 \} \cup \{ 0 \}$ and
$p(L_2) \subset \{ \xi + i\eta \in \cx \,\vert\, \eta < 0 \} \cup \{ 0 \}$. Furthermore, 
$p^{-1}(0) \cap (L_1 \cup L_2) = \{ 0 \}$, and Kallin's lemma concludes the proof. 
\end{proof}

Transversality of the intersection is not a necessary condition for local polynomial convexity of the union 
$L_1 \cup L_2$ of two Lagrangian submanifolds.

\begin{ex}\label{Ex1}
Consider in $\cx^2$ with the coordinates $z = x+iy$, $w=u+iv$ the Lagrangian submanifolds 
$L_1 = \rl^2_{(x,u)}$ and $L_2 = \{ (z,w): (x+ix^3,u+iu^3), (x,u) \in \rl^2 \}$.
Then $L_1 \cap L_2 = \{ 0 \}$, the manifolds $L_1$ and $L_2$ are tangent at the origin and  their union 
$L_1 \cup L_2$ is locally polynomially convex at the origin.
Indeed, consider the polynomial $p(z,w) = z^2 + w^2$. Then 
$p(L_1) = \{ \zeta \in \cx \,\vert\, \Re \zeta \ge 0, \Im \zeta = 0\}$, i.e., the positive real semi-axis. 
The image $p(L_2)$ is contained in the set $\{ \zeta \in \cx \,\vert\, \Im \zeta > 0 \} \cup \{ 0 \}$. 
Therefore, the polynomially convex hulls of these images intersect only at the origin which is a boundary 
point for both of them. It is easy to see that  $p^{-1}(0) \cap (L_1 \cup L_2) = \{ 0 \}$.
Hence, by Kallin's lemma $L_1 \cup L_2$ is locally polynomially convex near the origin.
\end{ex}

\section{Boundary behaviour of analytic discs near singular real manifolds}\label{s.5}

In this section we establish continuity up to the boundary of certain holomorphic discs, which is 
needed for the proof of Theorem~\ref{continuity10}, and discuss related results. From  Theorem~\ref{InterTheo} 
and Theorem~\ref{t:ga}  we obtain immediately

\begin{prop}\label{LocalGlobal}
Let $L$ be a smooth compact Lagrangian surface in $\cx^2$ with a finite number of transverse double self-intersections and generic open Whitney umbrellas. Then $L$ is locally polynomially convex.
\end{prop}

Let $K$ be a compact subset of $\cx^n$. There are several ways to define  an analytic disc with the boundary on $K$. The 
simplest one is to assume that the map $f$ is continuous  on the closed disc ${\overline\D}$ and $f(\partial\D)$ is contained 
in~$K$. In this case the image $X:= f(\D)$ is a complex purely 1-dimensional set in $\cx^n$ (i.e., a complex curve) with 
the boundary  $\partial X:= \overline X \setminus X$ contained in $K$. 

One can weaken the assumption of boundary continuity. Let $\gamma$ be a nonempty subset of $\partial\D$. By the 
{\it cluster set} $C(f,\gamma)$ of an analytic disc $f$ on $\gamma$ we mean the set of the (partial) limits of the sequences 
$(f(\zeta_k))$ for all sequences $(\zeta_k)$ in $\D$ converging to $\gamma$, i.e., such that 
$dist(\zeta_k,\gamma) \longrightarrow 0$. We recall a well -known (and easy to prove, see~\cite{CoLo}) fact that $C(f,\partial\D)$ is connected. If the cluster set $C(f,\partial \D)$ is contained in the compact set $K$, then $f:\D \setminus f^{-1}(K) \longrightarrow \cx^n \setminus K$ is a proper holomorphic map. Therefore, the image $X = f(\D)$ is still 
a complex curve with the boundary $\partial X = \overline X \setminus X$ contained in $K$ although the restriction 
$f\vert_{\partial\D}$ is not defined. The following theorem gives a sufficient condition for the equivalence of these two notions of the boundary of an analytic disc.

\begin{prop}\label{t:cont1}
Let $E$ be a compact subset of $\cx^n$ and $f$ be an analytic disc with $C(f,\partial\D) \subset E$. Suppose that $E$  is a smooth totally real submanifold in a neighbourhood of every point of $C(f,\partial\D)$ except possibly a finite subset $Sing(E) = \{ p_1,...,p_m\}$. Suppose further that $E$ is locally polynomially convex near  every point $p_j \in Sing(E)$.  Then  $f$ 
extends continuously to $\overline\D$.
\end{prop}

The case when the set $Sing(E)$ is empty is due to Chirka \cite{Ch}. In the proof of this theorem we employ the following result essentially due to  Forstneri\v c and Stout~\cite{FoSt}:

\begin{lemma}\label{l.fs}
In the assumptions of Proposition~\ref{t:cont1}, there exists a neighbourhood 
$\Omega$ of $C(f,\partial\D)$ in $\cx^n$ and a continuous non-negative plurisubharmonic function $\rho$  on $\Omega$ 
such that $E \cap \Omega = \{ p \in \Omega: \rho(p) = 0 \}$. Furthermore, for every $\delta > 0$ one can 
choose  $\rho$  such that it is a smooth strictly plurisubharmonic function on $\Omega \setminus \cup_{j=1}^m \B(p_j,\delta)$.
\end{lemma}
Forstneri\v c and Stout stated this result for a totally real disc with a finite number of hyperbolic points in $\cx^2$ 
but one can adapt their proof to the general situation with minor changes since it uses only local polynomial convexity 
near a hyperbolic point. For reader's convenience we provide the details.

\begin{proof}[Proof of Lemma~\ref{l.fs}.] Consider the function $\rho_1(z) = ({dist}(z,E))^2$, where ${dist}$ 
denotes the Euclidean distance. This function is Lipschitz continuous on a neighbourhood $\Omega_1$ of $E$.  We will construct a suitable  modification of $\rho_1$ near every point $p_j \in Sing(E)$ making it plurisubharmonic. Since the construction is local, we fix a point $p_j$ and assume that $p_j = 0$.

By assumption, for every  real $\e> 0$ small enough  the intersection $E \cap \e\overline\B$ is polynomially convex.  
Given $\delta>0$, fix numbers $\e$, $\e_1$ and $\e_2$ such that $0 < \e_2 < \e_1 < \e < \delta/2$. Since $E':= E \setminus Sing(E)$ is totally real near every point, there exists an open neighbourhood $\Omega_2 \subset \Omega_1$ of  $E'$ such that $\rho_1$ is a smooth strictly plurisubharmonic function on $\Omega_2$.

Consider  a smooth function $\psi: \cx^n \longrightarrow (-\infty,0]$ with the following properties: 
\begin{itemize}
\item[(a)] the support of $\psi$ is contained in $\e_1\B$,
\item[(b)] $\psi < 0$ in a neighbourhood of $E \cap \e_2\overline{\B}$,
\item[(c)] the $C^2$ norm of $\psi$ is small enough so that the function $\rho_2:=\rho_1 + \psi$ is strictly plurisubharmonic in a neighbourhood $\Omega_3$ of $E \cap (\e\overline\B \setminus \e_2\B)$.
\end{itemize}

Consider the function $\rho_3:= \max (\rho_2,0)$. It is nonnegative, continuous and plurisubharmonic on $\Omega_3$. Furthermore, $\rho_3$ vanishes in a neighbourhood of $E \cap (\e_2\partial\B)$. Therefore, one can extend the restriction of 
$\rho_3$ to $\Omega_3$ as a plurisubharmonic function on a neighbourhood $\Omega_4$ of $E \cap \e\overline{\B}$ by
setting it  to zero in a neighbourhood of $E \cap (\e\overline\B \setminus \e_2\B)$. We again denote this extended function 
by $\rho_3$. 

It follows from polynomial convexity of $E$ near the origin (see \cite[Thm 1.3.8, p. 25]{St}) that there exists a 
smooth nonnegative plurisubharmonic function $\phi$ on $\cx^n$ such that its zero locus coincides with $E \cap \e\B$. The function $\rho_4^t := \rho_3 + t\phi$ is nonnegative plurisubharmonic on $\Omega_4$ for every $t> 0$. It is easy to see that its zero locus coincides with $E \cap \e\overline\B$ for every value of the parameter $t> 0$. Since $\rho_3 = \rho_1$ on the set $\Omega_4 \cap (\e\overline\B \setminus \e_1\B)$, the equality  $\rho_4^t = \rho_1 + t\phi$ holds there. Hence $\rho_4^t$ is smooth and strictly plurisubharmonic on  $\Omega_4 \cap (\e\overline\B \setminus \e_1\B)$ for every $t > 0$.

Now fix a smooth function $\chi:\cx^n \longrightarrow [0,1]$ equal to zero on $\e_1\overline\B$ and to 1 outside $\e\B$. In order to patch $\rho_1$ and $\rho_4$, we set
$$
\rho^t = \chi\rho_1 + (1-\chi)\rho_4^t .
$$
On the set  $\Omega_4 \cap (\e\overline\B \setminus \e_1\B)$ where the patching occurs we have 
$\rho^t = \rho_1 + (1-\chi)t\phi$. Since $\rho_1$ is strictly plurisubharmonic on this set, we can fix $t> 0$ sufficiently small such that $\rho^t$ is strictly plurisubharmonic on  $\Omega_4 \cap (\e\overline\B \setminus \e_1\B)$. Dropping the upper 
index $t$, we obtain a function $\rho$ which is the required modification of $\rho_1$ near $p_j$. By repeating this argument near every point $p_j \in Sing(E)$, we conclude the proof of the lemma.
\end{proof}

The second ingredient is the following result \cite[Cor. 1.2]{ChCoSu}:

\begin{theorem}
\label{CCS}
Let $\Omega$ be a domain in $\cx^n$, $\rho$ a plurisubharmonic function in $\Omega$ with the zero set $X = \rho^{-1}(0)$, and $f:\D \longrightarrow \Omega^+ = \{ \rho \ge 0 \}$ a bounded analytic disc such that the cluster set $C(f,\gamma)$ on an open arc $\gamma \subset \partial \D$ is contained in $X$. Assume that for a certain point $\zeta \in \gamma$ the cluster set 
$C(f,\zeta)$ contains a point $p \in X$ such that, for some $\e > 0$, the function $\rho(z) - \e\parallel z \parallel^2$ is plurisubharmonic in a neighbourhood of $p$. Then $f$ extends to a H\"older $1/2$-continuous mapping in a neighbourhood of $\zeta$ on $\D \cup \gamma$.
\end{theorem}

\begin{proof}[Proof of Proposition~\ref{t:cont1}.] 
Fix $\zeta \in \partial\D$. Consider first the case when the cluster set $C(f,\zeta)$  contains at least one  point $p \in E \setminus Sing(E)$. 
Let $\rho$ be the plurisubharmonic function given by Lemma~\ref{l.fs}. 
Shrinking the balls $\B(p_j,\delta)$, one can assume that $\rho$ is strictly plurisubharmonic near $p$.  Then $f$ extends continuously to $\zeta$ by Theorem \ref{CCS}. Now suppose that $C(f,\zeta)$ contains only  points of $Sing(E)$. The cluster set $C(f,\zeta)$ is connected \cite{CoLo}, and so it can contain only one singular point $p_j \in Sing(E)$ which means that the analytic disc $f$ extends continuously to the point $\zeta$.
\end{proof}

\begin{cor}\label{LocalGlobalCor}
Let $L$ be a smooth compact Lagrangian surface in $\cx^2$ with a finite number of transverse double self-intersections and generic open Whitney umbrellas. Suppose that $f$ is an analytic disc such that the cluster set $C(f,\partial\D)$ is contained in $L$. Then $f$ extends continuously to $\overline\D$.
\end{cor}

The corollary immediately follows from Proposition~\ref{t:cont1} in conjunction with Theorems~\ref{t:ga} and~\ref{InterTheo}. 
The next result is a version of Gromov's removable singularities theorem with nonsmooth boundary data.

\begin{cor}\label{removal}
 Suppose that  a compact subset $E$ of $\cx^n$ is a totally real submanifold in a neighbourhood of every point of $E$  
 except a finite subset $Sing(E) = \{ p_1,...,p_m\}$ (which may be empty).  Assume that $E$ is locally polynomially 
 convex near  every point $p_j \in Sing(E)$. Consider  an analytic disc  $f$  of bounded area,  continuous on 
$\overline{\D} \setminus \{ 1 \}$ and such that $f(\partial\D \setminus \{ 1 \})$ is contained in $E$. 
Then $f$ extends continuously to $\overline\D$.
\end{cor}

\begin{proof}
By \cite[Thm 2]{Al}, $f:\D \setminus f^{-1}(E) \longrightarrow \cx^n \setminus E$ is a proper map. 
Therefore, $C(f,\partial\D)$ is contained in $E$ and we apply Proposition~\ref{t:cont1}. 
\end{proof}

In particular, the conclusion of this corollary holds if $E = L$ is a smooth compact Lagrangian surface in $\cx^2$ with a finite number of transverse double self-intersections and generic open Whitney umbrellas.

\begin{rem} Theorem~\ref{LocalGlobalCor} and Corollary~\ref{removal} still remain true if $E$ is a smooth surface
totally really embedded in $\cx^2$ with a finite number of hyperbolic points. Indeed, Forstneri\v c and Stout \cite{FoSt} 
proved that in this case $E$ is locally polynomially convex near every hyperbolic point.
\end{rem}

\begin{rem} Let $L$ satisfy the assumptions of Proposition~\ref{LocalGlobal}. Then every nonconstant analytic disc $f$  
with boundary attached to $L$ has area bounded away from zero by a constant depending only on $L$. 
Indeed, since $L$ is locally polynomially convex, 
there exists $\e > 0$ such that the boundary of $f$ cannot be contained in the ball $\B(p,\e)$. Then there exists a point 
$p$ in the boundary of $f$ such that $\B(p,\e/2)$ contains only smooth points of $L$. Applying estimates from \cite{Ch}
to the analytic set $f(\D) \cap \B(p,\e)$, we obtain the desired estimate from below. Combining this result with 
Corollary~\ref{removal}, we conclude that Gromov's compactness theorem holds for families of analytic discs with boundary 
glued to $L$. The same holds for totally real surfaces with hyperbolic points and for a sequence of analytic discs with 
uniformly bounded area.
\end{rem}

\section{Hulls of compact real manifolds with singularities}

\subsection{Hulls of Lagrangian surfaces with open umbrella singularities.}
In this subsection we give the proof of Theorem~\ref{t.hulls}(i). Note that here umbrella points are not assumed to be generic. 
We also notice that according to Givental \cite{G} any orientable compact surface of genus $g > 1$ admits a Lagrangian inclusion into $\cx^2$  with $2g-2$ open umbrella singularities and without self-intersections (topologically this inclusion is 
an embedding).

\begin{proof}[Proof of Theorem~\ref{t.hulls}(i)]
It follows from the definition of umbrella points that a neighbourhood of an open umbrella in $L$ is homeomorphic to the unit ball in $\rl^2$; hence $L$ is a topological manifold.  According to a  result of A.~Browder (for the orientable case) and Duchamp-Stout (for the non-orientable case), see \cite[Cor. 2.3.5, p. 95]{St}, no $n$-dimensional compact topological 
manifold in $\cx^n$ is polynomially convex. Applying it to $L$, we obtain $L\ne \widehat L$.

We employ the following  result due to Duval-Sibony \cite[Thm 5.3]{DuSi1}: let $K$ be a compact set in $\cx^n$ with 
$\widehat K \neq K$ and $Y$ be a closed polynomially convex subset of $K$. Suppose that $K \setminus Y$ is locally contained in a totally real manifold. Then there exists a compactly supported, positive (1,1) current $S$ such that 
$\supp (dd^c S) \subset K$ and $\supp (S)$ is not contained in $K$.

Since a finite set is polynomially convex, we can take $L$ as $K$ and the set of its umbrella points as $Y$, 
and use  $\widehat L \neq L$. 
\end{proof}

\begin{rem} Recall that a (1,1) - current on an open subset $\Omega$ is called pluriharmonic (or simply harmonic) in 
$\Omega$ if  $dd^cS = 0$ in $\Omega$. Thus, the current $S$ given by Theorem~\ref{t.hulls}(i) is pluriharmonic in 
$\cx^2 \setminus L$. The theory of pluriharmonic currents is now well developed, see, for instance,  
\cite{DinL,FoSi}.  
\end{rem}

\begin{rem} One can obtain more information about the structure of the current $S$ of Theorem~\ref{t.hulls}(i). 
We recall two general results concerning the structure of polynomially convex hulls, which we state for $L$, though 
they hold for any compact subset of $\cx^n$. Let $p$ be a point in $\widehat L \setminus L$. The first result, which 
is due to Duval-Sibony \cite{DuSi1}, states that there exists a positive $(1,1)$-current $R$ with $p \in \supp(R)$ such that 
\begin{eqnarray}
\label{DS1}
dd^c R = \sigma - \delta_p ,
\end{eqnarray}
 where $\sigma$ is a representative Jensen measure for evaluation at $p$, and 
$\delta_p$ is the Dirac mass at $p$.  A typical example of such a current  arises if there exists a bounded holomorphic disc 
$f: \D \longrightarrow \cx^2$ such that its radial boundary values $\tilde f(\theta):= \lim_{r \to 1} f(re^{i\theta})$ 
belong to $L$ for almost all $e^{i\theta} \in \partial\D$. Suppose for simplicity that $f(0) = p$. It is well-known that the 
image $f(\D)$ is contained in $\widehat L$. However, in general $f(\D)$ is not a complex analytic set and its current of integration is not defined. Consider the Green function $G(\zeta) = (-1/2\pi)\log \vert\zeta\vert$, $\zeta \in \D$ and the 
current $[\D]$. The current $G[\D]$ acts on a test form $\psi \in {\mathcal D}^{1,1}(\cx)$ by  
$$
\langle G[\D],\psi\rangle := \langle [\D],G\psi\rangle = \int_\D G \psi.
$$ 
Pushing it forward by the analytic disc $f$, we  obtain the {\it Green current} $G_f$ of the disc $f$ acting  on $\phi \in {\mathcal D}^{1,1}(\cx^2)$ by 
$$\langle G_f,\phi\rangle :=  \langle f_*G[\D],\phi\rangle = \langle G[\D],f^*\phi\rangle.$$
It is easy to check that this current is defined correctly and satisfies $dd^cS =  f_*\sigma - \delta_p$, and $ f_*\sigma$ is a Jensen measure for $p$, see \cite{DuSi1}.

Fix a Runge domain $\Omega$ containing $L$.  For every $\varepsilon > 0$ there exist a subset 
$\Gamma \subset \partial\D$ of measure less than $\varepsilon$ and a map $f:U \longrightarrow \Omega$ 
holomorphic in a neighbourhood $U$ of $\overline\D$ with $f(\zeta) \in L$ for all $\zeta \in \partial\D \setminus \Gamma$. 
This is  Poletsky's theorem \cite{Po}. Recently Wold \cite{W} proved that Poletsky's theorem implies the existence of 
Duval-Sibony's current $R$ satisfying (\ref{DS1}). He proves that every  such $R$  can be obtained as a limit of the Green currents $G_{f_k}$ corresponding to a sequence of Poletsky's discs. As the proof of \cite{DuSi1} shows, the current $S$ 
of Theorem~\ref{t.hulls}(i) is a limit of a sequence of normalized (i.e., with suitable positive factors) currents 
satisfying (\ref{DS1}); these currents are associated with a suitably chosen sequence of points $p_k$ converging to the 
smooth part of $L$. Combining this with the result of Wold, we conclude that  the current $S$ in Theorem~\ref{t.hulls}(i) 
is a limit of a suitably chosen sequence of (up to positive factors) Green currents $G_{f_k}$  corresponding to
 Poletsky's discs. 
\end{rem}

\begin{rem}
When $L$ has open Whitney umbrella points and transverse self-intersections one can remove self-intersections
by Lagrangian surgery, see \cite{LS}. This procedure consists of gluing an arbitrarily small Lagrangian handle to $L$
near the self-intersection point, turning $L$ into a local embedding. Applying such a surgery to every self-intersection
we obtain a topological Lagrangian embedding satisfying the assumptions of Theorem~\ref{t.hulls}(i). Note that this
operation changes the genus of $L$.
\end{rem}

\subsection{Hulls of totally real immersions} 
In the remaining part of the section we prove Theorem~\ref{t.hulls}(ii) and Theorem~\ref{continuity10}. Theorem~\ref{t.hulls}(ii) is due to Duval-Sibony \cite{DuSi2} 
in the case when $E$ is a totally real embedding. 

Notice that every compact $n$-manifold admits a totally real immersion into $\cx^n$ (see, for instance, \cite{Fo}).
Since immersions are not in general topological submanifolds of $\cx^n$, the algebraic topology methods 
used in the proof of polynomial nonconvexity in Theorem~\ref{t.hulls}(i) do not work directly. Instead, we use 
Alexander's  version~\cite{Al} of Gromov's method \cite{G}, which gives the existence of an analytic disc attached to 
a totally real immersion, which is stronger than the existence of Poletsky's discs. This yields more precise information 
on the structure of the polynomially convex hull of $L$ than gluing to $L$ a  closed positive current.

\begin{proof}[Proof of Theorem~\ref{t.hulls}(ii).] The current $S$ will also be obtained as a limit of (normalized) currents of integration over suitable analytic discs. The absence of umbrellas on $E$ allows us to choose analytic discs  with boundaries better attached  to $E$ than  Poletsky's discs.

A {\it nearly smooth} holomorphic disc of class $C^m$ is an $H^{\infty}$ (i.e., bounded holomorphic) disc  which 
extends $C^m$-smoothly  to $\partial \D\setminus \{1\}$. We say that a nearly smooth holomorphic disc $f$ is attached 
to a compact subset $K\subset  \cx^n$, if $f(\partial\D\setminus\{1\})\subset K$. If $f$ is {\it nonconstant} we  
call it an A-disc of class $C^m$ (after Herbert~Alexander, who proved the existence of such discs for totally real 
manifolds,~\cite{Al}). We simply write A-disc if it is of class $C^\infty$.

In the next subsection we will prove  Proposition~\ref{GluingDisc2} that  gives the existence of an  A-disc $f$ for a totally 
real immersed manifold $E$. Assuming this result we conclude the proof of Theorem~\ref{t.hulls}(ii). It suffices to apply 
to an A-disc $f$ the result of Duval-Sibony \cite[Thm 3.1]{DuSi2}. Though  their  theorem is stated for a totally real 
embedding $E$, the part of their proof which  we need goes  through in our case. Indeed,  consider a exhausting 
sequence $U_k$ of 
subdomains in $\D$ defined as in \cite{DuSi2}. Pushing forward  their currents of integration by the disc $f$, we obtain 
the sequence of currents  $f_*[U_k]$. Set   $a_k = area(f\vert_{U_k})$. Then the argument of \cite{DuSi2} shows 
that the current $S$ satisfying the hypothesis of the proposition is the limit of the sequence 
$S_k = f_*[U_k]/a_k$. This part of the argument of Duval-Sibony is based on a general estimate of the harmonic 
measure in \cite[Lem. 3.2]{DuSi2} and a general isoperimetric inequality for analytic discs \cite[p. 629]{DuSi2}. 
These ingredients do not require any assumptions on $E$. Then the convergence of the sequence $f_*[U_k]/a_k$ to $S$ 
follows by the argument of \cite[p. 629]{DuSi2}, which is independent of the structure of $E$ as well.  
The obtained current is supported on $f(\D)$, so its support is contained in $\widehat E$. It also follows from 
\cite[p. 629-630]{DuSi2} that $\supp(S) \not\subset E$.
\end{proof}

\subsection{Existence of A-discs and proof of Theorem~\ref{continuity10}.}\label{s.Adisc}

Let $E = (\tilde E,{\iota})$ be a pair which consists of a compact real manifold $\tilde E$
of dimension $n\ge 1$ and a $C^\infty$-smooth totally real immersion ${\iota}: \tilde E\to \cx^n$. We simply 
say that $E$ is an immersed totally real manifold in $\cx^n$ identifying it with the image ${\iota}(\tilde E)$. 
We say that an A-disc $f$ is {\it adapted} for the immersion $E$ if for every point $\zeta \in \partial \D \setminus \{ 1 \}$ there exist an open arc $\gamma\subset \partial \D$ containing $\zeta$ and a smooth map 
$ f_b:\gamma \longrightarrow \tilde E$ satisfying ${\iota} \circ f_b = f\vert_\gamma$. 
In other words, in a neighbourhood of every self-intersection point $p$ of $E$ the values of $f$ belong to a 
smooth component of $E$ through $p$.

\begin{prop}\label{GluingDisc2}
Let $E = (\tilde E,{\iota})$ be an immersed totally real manifold in $\cx^n$. Then 
\begin{enumerate}
\item[(i)] $E$ admits an adapted  A-disc $f \in C(\overline\D \setminus \{ 1 \})$.
\item[(ii)] If in addition $E$ is Lagrangian, then $f$ is of  bounded area with the cluster set $C(f,\partial\D)$ 
contained in $E$. Its image $X= f(\D)$ is a holomorphic curve of bounded area with the boundary 
$\partial X := \overline X \setminus X$ contained  in $E$.

\end{enumerate}
\end{prop}
 
It follows by the maximum principle that the disc $f$ is contained in the polynomially convex hull of $E$. Since the area of $X$ is finite, the current of integration $[X]$ over 
$X$ is correctly defined. While in Proposition~\ref{GluingDisc2} we do not impose any restrictions of transversality type 
on $E$, under additional assumptions we can deduce the boundary continuity of a disc. 

\begin{proof}[Proof of Theorem~\ref{continuity10}.] Theorem follows from Proposition~\ref{GluingDisc2} and Proposition~\ref{t:cont1}.
\end{proof}

It remains to establish Proposition~\ref{GluingDisc2}.

\begin{proof}[Proof of Proposition~\ref{GluingDisc2}.]
The proof follows Alexander's argument with some necessary modifications. We first deal with part (i).

\medskip

 {\it Step 1. Manifolds of discs and elliptic estimates.} Fix a point $p = {\iota}(\tilde p) \in E$ which is not a self-intersection point  and  fix also a  non-integer $r > 1$.
  Consider the set of pairs 
\begin{equation}\label{mnfd-U}
{\mathcal F} = \left\{(f,f_b)\in C^{r+1}(\D, \cx^n)\times  C^{r+1}(\partial \D, \tilde E): 
f(\partial \D) \subset E,\ f(1) = p,\ {\iota} \circ f_b = f|_{\partial \D} \right\}.
\end{equation}
In other words, together with a (not necessarily analytic) disc $f$ we specify a lift of its boundary  to the source manifold $\tilde E$. This idea is due to Ivashkovich-Shevchishin \cite{IvSh}. In what follows we will write shortly $f$ instead of $(f,f_b)$ when it does not lead to a confusion.

Denote by  $F$  an open subset of ${\mathcal F}$ which consists of $f$ homotopic to a constant map
 $f^0 \equiv p$ in ${\mathcal F}$. It is well-known that $F$ is a $C^\infty$-smooth complex Banach manifold.
Denote by $G$ the complex Banach space of all $C^r$ maps $g:\D \to \cx^n$. Set
$H = \{ (f,g) \in F \times G: \partial f /\partial \overline\zeta = g \}$.
Then $H$ is a connected submanifold of $F \times G$.  

For $0<t<1$, let $\D_t := t \D$, and $\D^+_t := t\D \cap \{\Im \zeta >0\}$.

\begin{lemma}\label{AlEst2}
Let $E'$ be an embedded totally real manifold. Let  $f_k:(\D_t^+,\partial \D_t^+ \cap \rl) \to (\cx^n,E')$ 
be maps of class $C^{r+1}$ which converge uniformly to $f:(\D_t^+,\partial \D_t^+ \cap \rl) \to (\cx^n,E')$. Suppose 
that the sequence $g_k = \partial f_k/\partial\overline\zeta$ converges in $C^r(\D^+_t)$ to $g \in C^r(\D^+_t)$. Then 
for every $\tau < t$ one has $f \in C^{r+1}(\D^+_t)$ and $\{ f_k \}$ converges to $f$ in $\D^+_\tau$ in the  $C^{r+1}$ norm.
\end{lemma}

Denote by
$$
T_\D f(\zeta) = \frac{1}{2\pi i}\int_\D \frac{f(\tau)}{\tau - \zeta}d\tau \wedge d\overline\tau
$$
the Cauchy-Green integral on $\D$. Recall the classical  {\it regularity property} of the Cauchy-Green integral: for every 
noninteger $s > 0$ the linear map $T_\D:C^s(\D) \longrightarrow C^{s+1}(\D)$ is bounded. The proof of 
Lemma~\ref{AlEst2} given in \cite{Al} is based on the standard elliptic "bootstrapping" argument employing this 
regularity property of the Cauchy-Green operator and elementary  estimates  of the harmonic measure. An immediate 
but crucial for our goals observation is   that this proof is {\it purely local}, i.e., all estimates and the convergence are established in a neighbourhood of a given boundary point of a disc. The global statement is the following 

\begin{lemma}\label{AlEst1}
Let $E$ be an immersed totally real manifold. Suppose that for a sequence $\{ (f_k, (f_k)_b)\}$ in ${\mathcal F}$ 
the sequence $\{f_k\}$ converges to a continuous mapping $f:(\D,\partial\D) \to (\cx^n,E)$ uniformly on 
$\overline \D$, and $g_k:= \partial f_k/\partial \overline\zeta$ converges in $C^r(\D)$ to $g \in C^r(\D)$. 
Then $f \in  C^{r+1}(\D)$, $\{(f_k)_b\}$ converges to some $f_b$, and $\{(f_k, (f_k)_b)\}$ converges to $(f, f_b)$ 
in ${\mathcal F}$ after possibly passing to a subsequence.
\end{lemma}

We stress that the local character of Lemma~\ref{AlEst2} allows us to pass automatically from an  embedded $E'$ to 
a globally immersed $E$ in Lemma~\ref{AlEst1}. Indeed, suppose that $q$ is a self-intersection point of $E$ and 
$f(\zeta_0) = q$ for some $\zeta_0\in\partial\D$. It follows from the uniform convergence of the sequence $(f_k)$ 
and the definition of the set ${\mathcal F}$ that there exists a neighbourhood $U$ of $\zeta_0$ such that 
$f(U \cap \partial\D)$ and after passing to a subsequence, $f_k(U \cap \partial\D)$ belong to the same smooth 
component through $p$ of the immersed manifold $E$. This reduces the situation to the embedded case of 
Lemma~\ref{AlEst2}.

Considering the finite covering of $\partial\D$ by such neighbourhoods we obtain $C^{r+1}$ convergence in a
neighbourhood of $\partial\D$. The convergence in the interior of $\D$ follows, since $f_k = T_\D g_k + h_k$ and the 
bounded sequence $(h_k)$ of holomorphic functions form a normal family. 

\medskip

{\it Step 2: Renormalization.} The canonical projection $\pi: H \to G$ given by $\pi(f,g) = g$ is a map of class $C^1$ 
between two Banach manifolds. It is known \cite{Al, Gr, IvSh} that $\pi$ is  a Fredholm map of index $0$ and the 
constant map $f^0$  is a regular point for $\pi$; the proof in the immersed case is the same as in the embedded case,
see \cite {IvSh}.

The crucial property of $\pi$ is described in the following lemma proved in \cite{Al}: {\it the map $\pi$ is not surjective.} 
Now assume by contradiction that an adapted  A-disc of class $C^{r+1}(\D)$ for $E$ does not exist.  In particular, 
$\pi^{-1}(0) = \{ f^0\}$. Then $0 \in G$ is a regular value of $\pi$. If $\pi$ is proper, then Gromov's argument based on 
Sard-Smale's theorem implies surjectivity of $\pi$ (see \cite{Al}) -- a contradiction. Thus, it remains to show that 
$\pi:H \to G$ is proper.

\smallskip

 Arguing by contradiction, suppose that $\pi$ is not proper. Then there exists a sequence $\{(f_k,g_k)\}\subset H$ 
 such that $g_k\to g$ in $G$ but $f_k$ diverge in $F$. For every $k$ consider the function $q_k$ defined by   
 $q_k(\zeta)= T_\D g_k(\zeta)$  for $\zeta \in \overline\D$ and $q_k(\zeta) = 0$  on $\zeta \in \cx\setminus \overline\D$.
 Then $q_k \to q = T_\D g$ in $C^{r+1}(\overline\D , \cx^n)$ and $f_k = q_k + h_k$, where 
$h_k\in C^{r+1}(\D , \cx^n)$ and $h_k$ is holomorphic on $\D$.  We have  $f_k(\partial \D)\subset E$ 
and $q_k$ are uniformly bounded since $g_k$ are; we conclude that $h_k|_{\partial \D}$ are uniformly  bounded. 
By the maximum principle the functions $h_k$ are uniformly bounded on $\overline\D$. Hence, $f_k$ are uniformly 
bounded.

Set $M_k = \sup_{\D}|h'_k(\lambda)|$. Since $h_k\in C^r(\D , \cx^n)$ and $r > 1$, 
the constants $M_k$ are finite for every~$k$. If $(M_k)$ contains a bounded subsequence, then a subsequence of 
$(h_k)$ converges uniformly on $\D$. Then a subsequence of $(f_k)$ converges uniformly, and by Lemma~\ref{AlEst1}
it converges in $F$ -- a contradiction. Thus, we may  suppose $M_k \to \infty$.  The key idea of 
\cite{Al} is to apply a renormalization argument.

There exists $\lambda_k\in \partial \D$ with $M_k= |h'(\lambda_k)|$ and, taking a subsequence if necessary, suppose that $\lambda_k\to \lambda^*$.
Set $z_k = (1-\frac{1}{M_k})\lambda_k\in \D$ and consider the renormalization sequence of Mobius maps  
$\phi_k(\lambda) = (\lambda + z_k)(1+\bar z_k\lambda)^{-1}$. Set $\tilde f_k = f_k\circ \phi_k$, 
$\tilde q_k = q_k\circ \phi_k$ and $\tilde h_k = h_k\circ \phi_k$. It is proved in \cite{Al} that after extracting 
a subsequence, the sequences $(\tilde q_k)$ and $(\tilde h_k)$ converge uniformly on compacts in 
$\overline\D\setminus \{-\lambda^*\}$ respectively to a constant map $c$  and a holomorphic map $\tilde h$.

 Notice that since $q_k$ converge in $C^{r+1} (\overline \D) $, the sequence $\tilde q_k$ converges  on compacts in 
$\overline\D\setminus \{-\lambda^*\}$ in this norm. Since Lemma~\ref{AlEst1} is local, it applies and gives the 
convergence of $(\tilde f_k)$ to $\tilde f$ also in the $C^{r+1}$-norm on compacts in 
$\overline\D\setminus \{-\lambda^*\}$. This is the key observation which makes Alexander's construction valid 
in the immersed case. 
Then again the argument of \cite{Al} shows that $\vert\tilde h_k'(\lambda_k)\vert$ converges to  
$1/2 = \vert \tilde h'(\lambda^*) \vert$. Hence, $\tilde f$ is nonconstant, and so is an adapted A-disc of class 
$C^{r+1}$. Since locally $E$ is an embedding and the disc is adapted, it is $C^\infty$ smooth on 
$\overline\D \setminus \{ 1 \}$ by the boundary regularity theorem for analytic discs.  This contradicts our 
assumption of nonexistence of adapted A-discs, which proves Proposition~\ref{GluingDisc2}(i).

\medskip

As for (ii), repeating verbatim the argument of \cite[p. 140-141]{Al} we obtain that the constructed in (i)  A-disc $f$ has a 
bounded area. Then \cite[Thm 2]{Al} implies that $f: \D\setminus f^{-1}(E) \to \cx^n \setminus E$ is a proper 
map. This completes the proof. 
\end{proof}


\end{document}